\title{\vspace{-.5in}Bounds for  elimination of unknowns in systems of differential-algebraic  equations\vspace{-0.1in}}
\documentclass[11pt,letter]{extarticle}
\usepackage[numbers]{natbib}
\usepackage[english]{babel}
     \usepackage[dvipsnames]{xcolor}
     \usepackage[leqno]{amsmath}
\usepackage{amssymb,amsthm,pslatex,color,bm}
\usepackage[letterpaper,margin=1in]{geometry}
\usepackage{diagbox}
\usepackage{comment,enumitem}
\usepackage[ruled,vlined,linesnumbered]{algorithm2e}
\usepackage{caption} 
\usepackage[colorlinks, linkcolor=reference,
            citecolor=citation,
          urlcolor=e-mail]{hyperref}

\makeatletter
 \newcommand{\linkdest}[2]{\Hy@raisedlink{\hypertarget{#1}{#2}}}
\makeatother

\definecolor{e-mail}{rgb}{0,.40,.80}
\definecolor{reference}{rgb}{.20,.60,.22}
\definecolor{citation}{rgb}{0,.40,.80}
\newtheorem{lemma}{Lemma}
\newtheorem{proposition}{Proposition}
\newtheorem{theorem}{Theorem}
\newtheorem*{theorem_cite}{Theorem}
\newtheorem{corollary}{Corollary}

\theoremstyle{definition}
\newtheorem{definition}{Definition}
\newtheorem*{remark}{Remark}
\newtheorem{remark_num}{Remark}
\newtheorem{example}{Example}

\DeclareMathOperator{\Ker}{Ker}

\DeclareMathOperator{\alg}{alg}
\date{}

\author{Alexey Ovchinnikov, Gleb Pogudin\footnote{Gleb Pogudin's prior addresses: Johannes Kepler University, Institute for Algebra; New York University, Courant Institute of Mathematical Sciences; Higher School of Economics (Moscow), Department of Computer Science}, and Thieu N. Vo\footnote{Thieu Vo's prior address: Johannes Kepler University, RISC} \smallskip\\ \small
CUNY Queens College, Department of Mathematics,
65-30 Kissena Blvd, Queens, NY 11367, USA \\ \small
CUNY Graduate Center, Department of Mathematics, 365 Fifth Avenue,
New York, NY 10016, USA\\ \small
\href{mailto:aovchinnikov@qc.cuny.edu}{aovchinnikov@qc.cuny.edu} \vspace{0.05in} \\  \small LIX, CNRS, \'Ecole Polytechnique, Institute Polytechnique de Paris, Palaiseau, France \\ \small
\href{mailto:gleb.pogudin@polytechnique.edu}{gleb.pogudin@polytechnique.edu}\\ \small Fractional Calculus, Optimization and Algebra Research Group, Faculty of Mathematics and Statistics,\\ 
\small Ton Duc Thang University, Ho Chi Minh City, Vietnam\\
\small\href{mailto:vongocthieu@tdtu.edu.vn}{vongocthieu@tdtu.edu.vn}}
\begin{document}
\maketitle
\vspace{-0.4in}
\begin{abstract}  
Elimination of unknowns in systems of equations, starting with Gaussian elimination, is a problem of general interest. 
The problem of finding an a priori upper bound for the number of differentiations in elimination of unknowns in a system of  differential-algebraic equations (DAEs) is an important challenge, going back to Ritt (1932).
The first characterization of this via an asymptotic analysis is due to Grigoriev's result (1989) on
quantifier elimination in differential fields, but the challenge still remained.

In this paper, we present a new bound, which is a major improvement over the previously known results.
We also present a new lower bound, which shows asymptotic tightness
of our upper bound in low dimensions, which are frequently occurring in applications.
Finally, we discuss applications of our results to designing  
new algorithms for  elimination of unknowns in systems of DAEs.
\end{abstract}
\vspace{-0.19in}
\section{Introduction}
\vspace{-0.07in}
Consider a system of  equations  (e.g., linear, polynomial, differential)
\vspace{-0.05in}
\begin{equation}\label{eq:main}
  f_1(\bm{x}, \bm{y}) = \ldots = f_N(\bm{x}, \bm{y}) = 0\vspace{-0.05in}
\end{equation}
 in two sets of unknowns, $\bm{x}$ and $\bm{y}$.
To \emph{eliminate} the $\bm{x}$-variables is to find, if it exists, a
nontrivial 
equation
$g(\bm{y}) = 0$  involving only the $\bm{y}$-variables that holds for every solution of~\eqref{eq:main} (a stronger version of the problem is to describe all such equations).  
Elimination of unknowns for systems of equations of different types,
starting from Gaussian elimination for linear systems,
is a classical problem.
In this paper, we study elimination of unknowns in systems of differential-algebraic equations (DAEs), existing applications of which include combinatorics~\cite{BBM2017}, mathematical analysis of dynamic models~\cite{Boulier2007,Juan2008,HOPY,Jiafan2009}, and control theory~\cite{Diop1989,Diop1991}.

The first theoretical method for elimination of unknowns in systems of DAEs was developed in~\cite[\S 67]{Ritt} by Ritt,  the founder of differential algebra.
The method can be viewed as a 
far reaching
generalization of  Gaussian elimination and was further developed, e.g., in~\cite{BLOP,HubertDiff}.
Ritt also proposed another approach~\cite[\S87-88]{Ritt}, which is similar to the prolongation-relaxation strategy used in 1847-48 by
Cayley \cite{Cayley1847,Cayley1848} and later by Macaulay~\cite[Chapter~I]{Macaulay} for polynomial equations.
Their technique
was to reduce elimination in a system of polynomial equations~\eqref{eq:main} to elimination in a system of linear equations via an upper bound $B$ such that
\vspace{-0.1in}
\begin{enumerate}[leftmargin=0.6cm,itemsep=-0.1cm,label=(\alph*)]
  \item considering the {\em prolongation}\vspace{-0.03in}
\begin{equation}\label{eq:prolpol}
\bm{x}^{\bm{\alpha}}\bm{y}^{\bm{\beta}}f_i(\bm{x},\bm{y})=0,\quad 1\leqslant i\leqslant N,\ \ |\bm{\alpha}|+|\bm{\beta}|\leqslant B,
\end{equation}
\item polynomial elimination of $\bm{y}$ in~\eqref{eq:main} is possible if and only if $\bm{y}$ can be eliminated in~\eqref{eq:prolpol} considered as a linear system in the monomials in $\bm{x}$ and $\bm{y}$ appearing in~\eqref{eq:prolpol} using Gaussian elimination ({\em relaxation}).
\end{enumerate}
Extending this idea, the approach to elimination of unknowns in a system~\eqref{eq:main} of DAEs
proposed by Ritt~was:
\begin{enumerate}[leftmargin=0.6cm,itemsep=-0.1cm,label=(\alph*)]
  \item \emph{Prolongation}: for a non-negative integer $B$, consider the derivatives 
  \vspace{-0.05in}
\begin{equation}\label{eq:diffB}
  f_i(\bm{x}, \bm{y}) =0,\; f_i(\bm{x}, \bm{y})' =0,\; \ldots,\; f_i(\bm{x}, \bm{y})^{(B)} =0,\;\; 1\leqslant i\leqslant N.\vspace{-0.08in}
  \end{equation}
  \item \emph{Relaxation}: apply polynomial elimination (for example, using~\cite[\S 55-60]{Ritt}) to~\eqref{eq:diffB} viewed as 
  polynomial equations in $\bm{x},\bm{x}',\ldots,  \bm{y},\bm{y}',\ldots$.
\end{enumerate}
The results of Ritt~\cite[\S87-88]{Ritt} imply that, for every system~\eqref{eq:main} of DAEs, the integer $B$ can be chosen large enough so that, if an elimination of $\bm{x}$ for~\eqref{eq:main} is possible, it can be performed using polynomial  elimination applied to~\eqref{eq:diffB}.
Thus, Ritt posed the following challenge in 1932~\cite[p.~118]{Ritt}, 
\begin{itemize}[leftmargin=2.95cm]\item[(Ritt's Challenge)]\em For the above process to become a genuine method of decomposition,  it would be necessary to have a method for determining permissible integers $B$.\end{itemize}
Since then, finding a bound for $B$ has been a major problem.
One of the classical results in model theory of differential fields is that the theory of differentially closed field of characteristic zero, $\mathrm{DCF}_0$,
    has quantifier elimination~\cite[Theorem~2.4]{MarkerMTDF}. 
    Using an algorithm for quantifier elimination as a black box, one can 
    solve the elimination problem, which can be encoded as an elimination of existential quantifiers for the unknowns to be eliminated.
    An asymptotic analysis for the computational complexity  of quantifier elimination in the case of constant coefficients
    was established by Grigoriev 
    in~\cite{Grigoriev}, more than 50 years after Ritt had posed the problem.
   The complexity was shown to be bounded by an expression triple-exponential in the number of variables to be eliminated, which also involved the number of other variables, the number of equations, and the size of coefficients.  
   Thus, this analysis did not give an explicit bound yielding a reasonable algorithm, so the challenge remained.  Yet not addressing the challenge, in the special case of $\bm{x}=\varnothing$, there has been progress, described in \hyperlink{relres}{Related Results}. 
 
We have overcome Ritt's challege, and our upper bound for $B$ 
 in Ritt's prolongation-relaxation process for elimination is 
 of the form (see Theorem~\ref{th:elimination_variables} for more details and Theorem~\ref{th:FullElimination} for 
a stronger version of the elimination problem):\vspace{-0.05in}
    \[
    d^{(\overline{m} + 1)2^{m + 1}},\vspace{-0.13in}
    \]
 where\vspace{-0.07in}
\begin{itemize}[leftmargin=0.5cm,itemsep=-0.04cm]
\item  $\bm{f}(\bm{x},\bm{y}) = 0$ is a system of DAEs,
\item $n = |\bm{x}|$ is the number of unknowns to be eliminated (not the total number of unknowns), 
\item
\hypertarget{def_d}{}
$h$ is the order of $\bm{f}$ in $\bm{x}$ and $d \geqslant 2$ the degree of $\bm{f}$  in $\bm{x},\bm{x}',\ldots,\bm{x}^{(h)}$,
\item\hypertarget{def_KV}{}
   $m$ and $\overline{m}$ are the dimension and codimension
    of the variety $V$ defined by $\bm{f}(\bm{x}, \bm{x}', \ldots, \bm{x}^{(h)}, \bm{y}, \bm{y}', \ldots) = 0$
    in the affine space $\mathbb{V}$
    of dimension $n(h + 1)$ 
    with coordinates $\bm{x}, \bm{x}', \ldots, \bm{x}^{(h)}$ (over the field $K$ of rational functions in $\bm{y},\bm{y}',\ldots$).
    \end{itemize}
    \vspace{-0.08in}
 The bound is polynomial in the degrees, exponential in the codimension, and doubly exponential in the dimension.

    Furthermore, if the polynomial ideal generated by $F$
    is radical, then the bound is significantly better (see Theorem~\ref{thm:radical_elimination}): \vspace{-0.05in} 
   \[
   \sum\limits_{i = 0}^m D^{2(2^{i} - 1)},
   \] 
   where $D$ is the degree of $V$ (see~\cite[p. 246]{Heintz}). 
   Concrete systems of differential equations arising in applications usually have this property, 
    and 
     many of them
    have $m = 0, 1$. 
  For instance, if the parameter identifiability problem of ODE models is approached via input-output equations, then one solves an elimination problem for a prime (and therefore radical) differential ideals (see \cite{LG94,OllivierPhD,allident,OPT19,ident-compare,Saccomani2003,meshkat2018} and the references therein). 
    The corresponding value of $m$ is equal to  $s - 2\ell$ (using Remark~\ref{rem:prolong}), where $s$ is the number of state variables and $\ell$ is the number of output variables (terminology/setup of the problem).
    Example~\ref{ex:LV} is a natural example with the resulting $m$ being $0$, more example can be found in~\cite[Appendix B]{SIAN} (half of benchmarks there have $m = 0, 1$).
    Examples~\ref{ex:vanderpol}  and~\ref{ex:pendulum} illustrate differential elimination problems in other contexts.
    
    If $m = 0$, then  the bound given by Theorem~\ref{thm:radical_elimination} is $1$, which is tight. 
    If $m = 1$, then  the bound given by Theorem~\ref{thm:radical_elimination} is at most\vspace{-0.08in}
    \[
      D^2 + 1. \vspace{-0.05in}
    \]
Our new lower bound for $m=1$ is 
$
\binom{D + 2}{2} - 1 = D^2/2+3D/2
$
(see Proposition~\ref{prop:lower_bound}), and so our upper bound is asymptotically tight for $m\leqslant 1$. 

   A bound for full elimination, which is finding all possible results of elimination of given order, 
   is presented  in Theorem~\ref{th:FullElimination}.

    Finally, we show how our bound can be used to design a  randomized (Monte Carlo) algorithm with guaranteed probability of correctness: given $0 < p < 1$, the algorithm decides whether an elimination of unknowns is possible with probability at least~$p$ (see Section~\ref{sec:probability}). 
    The implementation and examples are available at~\url{https://github.com/pogudingleb/DifferentialElimination.git}.
    
    In the remainder of the introduction, we present an outline the approach and difficulties to overcome, as well as discuss related results.

    \vspace{-0.12in}
\subsection*{Outline of the approach}   \vspace{-0.05in}
The conceptual flow of the derivation of the main results is as follows\footnote{Even though this derivation can also be viewed as a computational procedure, we are not suggesting to use this as an algorithm in practice (see Section~\ref{sec:probability} for an actual algorithm).}:\vspace{-0.07in}
\begin{enumerate}[leftmargin=0.5cm,itemsep=-0.05cm]
  \item  We reduce the case of a general system of DAEs to the case in which the system of DAEs generates a radical equidimensional (i.e., all  prime components have the same dimension)
  ideal of the polynomial ring $\hyperlink{def_KV}{K[\mathbb{V}]}$ (Section~\ref{secsub:arbitrary}).
  \item We then reduce the latter case to the case in which the system of DAEs generates a prime polynomial ideal $I$ (Section~\ref{secsub:radequidim}).
  \item The bound for the case of prime ideals is derived using the following divide-and-conquer approach (Section~\ref{secsub:prime}) with induction on $m := \dim I$:\vspace{-0.1in}
  \begin{enumerate}[leftmargin=0.7cm,itemsep=0.01cm]
    \item In the base case $m = 0$, the ideal $I$ is maximal.
    Then Lemma~\ref{lem:prolongation} implies that either $\sqrt{I^{(\infty)}} \cap K[\mathbb{V}] = I$, so the bound is 0, or $\sqrt{I^{(\infty)}} \cap K[\mathbb{V}] = I^{(1)} \cap K[\mathbb{V}] = K[\mathbb{V}]$, so the bound is 1.
    \item Suppose now that $m > 0$. If $I^{(1)}\cap K[\mathbb{V}] = I$, then the bound is again $0$ by Lemma~\ref{lem:prolongation}, and we are done with this prime component. Otherwise, we proceed as follows:
    \begin{enumerate}
      \item The key ingredient, Lemma~\ref{lem:one_polynomial}, implies that 
      there exists a polynomial $g\in I^{(1)}\cap K[\mathbb{V}]$ with $\deg g\leqslant D:=\deg I$ such that $\dim\langle I,g\rangle <\dim I$ (getting this degree bound is one of the main subtleties, thus providing a key improvement of the method used in~\cite{HrushovskiPillay,AlJerSol}).
      \item We pass to $\sqrt{\langle I, g \rangle}$ using Lemmas~\ref{lem:NoetherExponentC} and~\ref{lem:product}.
      \item Since all prime components of  
      $\sqrt{\langle I, g\rangle}$ are of dimension $m - 1$ 
      (also the sum of their degrees is at most $D^2$), we apply the argument inductively to each of them.
      \item The bounds for the prime components are combined together using Lemma~\ref{lem:intersection}.\vspace{-0.05in}
    \end{enumerate}
  \end{enumerate}
  \item The above steps yield a general bound given in Proposition~\ref{prop:any}.
  The main results are deduced from the proposition as follows:\vspace{-0.1in}
  \begin{itemize}[leftmargin=0.5cm,itemsep=-0.01cm]
    \item Theorem~\ref{thm:radical_elimination} follows from the proposition by restriction to radical ideals.
    \item Theorems~\ref{th:elimination_variables} and~\ref{th:FullElimination} are derived from the proposition by estimating the geometric data in terms of the combinatorial data (e.g., $D \leqslant \hyperlink{def_d}{d}^{\hyperlink{def_KV}{\overline{m}}}$, where $\overline{m}$ is the codimension of the corresponding variety).
  \end{itemize}
\end{enumerate}

\vspace{-0.05in}
We derive the asymptotic tightness of our bound
for $m = 1$ by finding a witness (for a quadratic lower bound)
of the form $x'=1$, $y'=y$, $P(x,y)=0$, with $\deg P\leqslant D$, that nevertheless has an ``approximate solution'' $x(t)=t$, $y(t)=e^t$ 
(that is, the equations in the system vanish at $t = 0$ up to order $\binom{D + 2}{2} - 1$ after substituting $(t, e^t)$).

We  derive a randomized (Monte Carlo) algorithm with guaranteed probability of correctness as follows: \vspace{-0.25in}
\begin{itemize}[leftmargin=0.5cm,itemsep=-0.05cm]
  \item Theorems~\ref{th:elimination_variables} and~\ref{thm:radical_elimination} reduce determining the possibility of elimination for a system of DAEs to determining the possibility of elimination for a polynomial system  in 
  $q$
  unknowns  
  $\bm{z} = \big(\bm{x}, \bm{x}',\ldots, \bm{x}^{(B_1)}\big)$ and $r$ unknowns $\bm{w} = \big(\bm{y}, \bm{y}',\ldots, \bm{y}^{(B_2)}\big)$ for suitable $B_1$ and $B_2$. 
  \item An elimination of the $\bm{z}$ variables for a system $\bm{p}(\bm{z}, \bm{w}) = 0$ of polynomial equations is possible if and only if the projection $\pi$ of the variety $X \subset\mathbb{A}^q\times\mathbb{A}^r$ defined by $\bm{p}(\bm{z}, \bm{w}) = 0$ to the $\bm{w}$-coordinates is not dominant.
  \item We check the dominance of $\pi$
  by determining whether the fiber over a random point on the $\bm{w}$-plane is not empty (cf.~\cite{RSV2018}).
The dimension $r$ of the search space 
  is bounded by Theorems~\ref{th:elimination_variables} and~\ref{thm:radical_elimination}.
  If every coordinate of a random point is sampled from a finite set $S$ (e.g., a finite set of integers), then the nonemptyness of the fiber is equivalent to the dominance of $\pi$
  with probability at least\vspace{-0.08in}
  \[
    1 - \deg X/|S|.\vspace{-0.12in}
  \]
  We show this by proving, in particular, that\vspace{-0.07in}
  \[
  \overline{\pi(X)}=\mathbb{A}^r\implies\frac{|S^r\cap Z|}{|S^r|}\leqslant \frac{\deg X}{|S|},\quad Z := \mathbb{A}^r\setminus \pi(X).
  \]\hypertarget{relres}{}
\end{itemize}
\vspace{-0.32in}
    \subsection*{Related results}\label{subsec:related_results}

    There are  related bounds for other problems about systems of DAEs:
    \vspace{-0.05in}\begin{itemize}[leftmargin=0.5cm,itemsep=-0.05cm]
      \item \textbf{Determining consistency.}
      To determine the consistency of a system of DAEs using the prolongation-relaxation strategy (also referred to as effective differential Nullstellensatz) is a special case of elimination in systems of DAEs because a system of DAEs is inconsistent if and only if it is possible to eliminate all of the unknowns (i.e., to derive a consequence of the form $1 = 0$). 
There has been significant progress in analyzing this problem~\cite{Seidenberg,GolubitskyEtAl,AlJerSol,Gustavson,TS2019}. However, it has been a challenge to find practical upper bounds for this problem, as the upper bounds obtained there
\vspace{-0.05in}
  \begin{itemize}[leftmargin=0.4cm,itemsep=-0.05cm]
  \item either are asymptotic 
  and so cannot be used in a differential elimination algorithm directly, 
  \item or have values that make them impossible to be used even for small examples.
  \end{itemize}\vspace{-0.05in}
  Our results address both issues for DAEs for the consistency problem.
       \item \textbf{Differential resultants}  can be used to give a  solution to the elimination problem  of generic systems of DAEs of a special form (see \cite{RuedaSendra2010,Rueda2013,GaoLiYuan2013,LiYuanGao2015,LiYuan2019} and the references given there).
      \item \textbf{Counting solutions.}
      Unlike in usual applications to modeling and sciences, some systems of DAEs arising in algebraic number theory (see, e.g., \cite[Section~5]{HrushovskiPillay} and \cite[Sections~5.1-5.2]{FreitagScanlon}) have only finitely many solutions, and an important problem is estimate this number.
      Such bounds were obtained and applied to number-theoretic problems in~\cite{HrushovskiPillay,FreitagSanchez,Gal,FreitagScanlon}.
      Theorem~\ref{th:FullElimination} can be used to design a prolongation-relaxation algorithm for determining the number of solutions of a given DAE (see Remark~\ref{rem:count}).
    \end{itemize}
    
\vspace{-0.2in}
\section{Preliminaries and main results}\label{sec:statements}
\vspace{-0.05in}
\subsection{Differential Algebra}\label{subsec:prelim_diff_alg}

	Throughout the paper, all fields are assumed to be of characteristic $0$.
	Let $R$ be a commutative ring.
	\begin{definition}[Differential rings]
    \begin{itemize}[leftmargin=0.5cm,itemsep=-0.05cm]
    \item[]
    \item
    	A map $D\colon R \to R$ satisfying $D(a + b) = D(a) + D(b)$ and $D(ab) = aD(b) + D(a)b$ for all $a, b \in R$ is called a \textit{derivation}.
        \item 
	    A \textit{differential ring} $R$ is a ring with a specified derivation $D$.
	In this case, we will denote $D(x)$ by $x^{\prime}$ and $D^n(x)$ by $x^{(n)}$.
    \item
	A differential ring that is a field will be called a \textit{differential field}.
   \item
    	A differential ring $A$ is said to be a \textit{differential $k$-algebra} over a differential field $k$ if $A$ is a $k$-algebra and the restriction of the derivation of $A$ on $k$ coincides with the derivation on $k$.
        \item    \hypertarget{diff_poly}{} 	Let $A$ be a differential $k$-algebra.
        \begin{itemize}[leftmargin=0.4cm,itemsep=0.05cm]
        \item 
	    We consider the polynomial ring $A\big[x^{(0)}, x^{(1)}, x^{(2)}, \ldots\big]$, where $x^{(0)}, x^{(1)}, x^{(2)}, \ldots$ are algebraically independent variables.
        We will also use the notation $x, x', x''$ for $x^{(0)}, x^{(1)}, x^{(2)}$, respectively.
        \item For $h\geqslant 0$, the polynomial algebra  $A\big[x^{(0)},x^{(1)},\ldots,x^{(h-1)}\big]$ is denoted by $A[x_h]$.
        \item Extending, for a tuple $\bm{\alpha}=(\alpha_1, \ldots, \alpha_n)\in\mathbb{Z}_{\geqslant 0}^n$ and variables $\bm{x}=(x_1,\ldots,x_n)$, the corresponding polynomial algebra is denoted by $A[\bm{x}_{\bm \alpha}]$.
	\item    Extending the derivation from $A$ to $A\big[x^{(0)}, x^{(1)}, x^{(2)}, \ldots\big]$ by $D(x^{(i)}) = x^{(i + 1)}$, we obtain a differential algebra.
	\item    This algebra is called the \textit{algebra of differential polynomials} in $x$ over $A$ and denoted  by $A[x_\infty]$.
	  \item  Iterating this construction, we define the algebra of differential polynomials in variables $\bm{x} := x_1, \ldots, x_n$ over $A$ and denote it by $A[\bm{x}_{\bm{\infty}}]$.
     If $A$ is a field, then the field of fractions of $A[\bm{x}_{\bm{\infty}}]$ is denoted by $A(\bm{x}_{\bm{\infty}})$.
        \end{itemize}
        \end{itemize}
	\end{definition}
	\begin{definition}[Ideals]\label{def:ideals}\hypertarget{def_ideals}{}
    \begin{itemize}[leftmargin=0.5cm,itemsep=-0.05cm]
    \item[]
    \item The ideal of a ring $R$ generated by $a_1,\ldots,a_n\in R$ will be denoted by $\langle a_1,\ldots,a_n\rangle$.
    \item
    	An ideal $I$ of a differential ring $R$ is said to be a \textit{differential ideal} if $a^{\prime} \in I$ for all $a \in I$.
	\item	The differential ideal generated by $a_1, \ldots, a_n \in R$ will be denoted by \hypertarget{infty_ideal}{$\langle a_1, \ldots, a_n\rangle^{ (\infty)}$}.
    \item For an ideal $I$  (not necessarily differential) of $k[\bm{x}_{\bm\infty}]$,
     $I^{(h)}$ denotes the ideal generated by all elements of the form $a^{(j)}$, where $a \in I$ and $j \leqslant h$. 
     If $h=\infty$, then $I^{(h)}$ denotes $\langle I \rangle^{(\infty)}$.
     \item An ideal $I$ is \textit{radical} if, whenever $a^n \in I$ for some $n > 0$, $a \in I$. The smallest radical ideal containing $a_1, \ldots, a_n$ will be denoted by $\sqrt{\langle a_1, \ldots, a_n\rangle}$. 
     \item For an ideal $I$ and a nonnegative integer $i$, the {\em equidimensional component} of $I$ of dimension $i$ is the intersection of prime components of $\sqrt{I}$ of dimension $i$.
     \item For a variety $X$, $\deg X$  denotes the degree of $X$ (see~\cite[Definition~1 and Remark~2]{Heintz}).
     \end{itemize}
    \end{definition}

 The following is a version of Hilbert's Nullstellensatz for  DAEs, which shows the correctness of the prolongation-relaxation approach to elimination for systems of DAEs.

   \begin{theorem_cite}[{\cite[Theorem~IV.2.1]{Kol}}]
		For all $f_1, \ldots, f_N \in k[{\bm x_{\bm \infty}}, \bm{y}_{\bm{\infty}}]$ and $g \in k[\bm{y}_{\bm{\infty}}]$, the following are equivalent\vspace{-0.05in}
        \begin{enumerate}[label = (\alph*),leftmargin=0.7cm,itemsep=-0.25cm]
          \item for every $(\bm{x}^\ast, \bm{y}^\ast)$ in every differential field extension of $k$,\vspace{-0.05in}
          \[
            f_1(\bm{x}^\ast, \bm{y}^\ast) = \ldots = f_N(\bm{x}^\ast, \bm{y}^\ast) = 0 \implies g(\bm{y}^\ast) = 0;
          \]
          \item there exists $M$ such that $g^M \in \langle f_1, \ldots, f_N \rangle^{(\infty)}$.
        \end{enumerate}
	\end{theorem_cite}

\subsection{Main result}\label{subsec:main_results}

In this section, we state our main results, and their consequences. Proofs are postponed until Section~\ref{sec:proofs}.

\begin{theorem}[Bound for an elimination]\label{th:elimination_variables}
	For all integers $s, t \geqslant 0$,  
    tuples $\bm{\alpha} = (\alpha_1, \ldots, \alpha_{s})\in \mathbb{Z}_{\geqslant 0}^{s}$,  and  $F\subset k(\hyperlink{diff_poly}{\bm{y}_{\bm\infty}})[\hyperlink{diff_poly}{\bm{x}_{\bm{\alpha}}}]
    $,  
    \[
    \langle F \rangle^{\hyperlink{def_ideals}{(\infty)}} \cap k[\bm{y}_{\bm\infty}] =  \{0\} \iff \langle F\rangle^{ \hyperlink{def_ideals}{(B)} } \cap k[\bm{y}_{\bm\infty}] =\{0\},\vspace{-0.05in}
    \]
    where\vspace{-0.03in}
   \begin{itemize}[leftmargin=0.5cm,itemsep=0.02cm]
   \item
     $ \label{eq:defN}
B = 
\begin{cases}  
	d^{(|\bm{\alpha}| - m + 1) 2^{m + 1}} ,& d \geqslant 2,\\
	m+1 , &d=1,
\end{cases}
$
\item $|\bm{\alpha}|= \alpha_1 + \ldots + \alpha_s$,
\item 
$\bm{x} := (x_1, \ldots, x_{s})$, $\bm{y} :=(y_1, \ldots, y_{t})$,

    	\item 
        $d= \max\limits_{f \in F} \deg_{\bm{x}} f$,
        
        \item $m=\dim\langle F\rangle$ in $k(\bm{y}_{\bm\infty})[\bm{x}_{\bm{\alpha}}]$.
    \end{itemize}
\end{theorem}

For many systems arising in applications, the ideal generated by $F$ turns out to be radical
(see examples in Section~\ref{sec:examples_numbers}). In this situation,  we present an improvement to Theorem~\ref{th:elimination_variables},      
in Theorem~\ref{thm:radical_elimination}.
It follows from our proofs, that this  new upper bound is always smaller that one provided by Theorem~\ref{th:elimination_variables}.

\begin{theorem}[Bound for an elimination for radical ideals]\label{thm:radical_elimination}
   For all integers $s, t \geqslant 0$,  
    tuples $\bm{\alpha} = (\alpha_1, \ldots, \alpha_{s})\in \mathbb{Z}_{\geqslant 0}^{s}$,  and  $F\subset k(\hyperlink{diff_poly}{\bm{y}_{\bm\infty}})[\hyperlink{diff_poly}{\bm{x}_{\bm{\alpha}}}]
    $,
	if the ideal $\langle F\rangle$ of $k(\bm{y}_{\bm\infty})[\bm{x}_{\bm{\alpha}}]$ 
    is \hyperlink{def_ideals}{radical}, then
    \[
    \langle F \rangle^{\hyperlink{def_ideals}{(\infty)}} \cap k[\bm{y}_{\bm\infty}] =  \{0\} \iff \langle F\rangle^{\hyperlink{def_ideals}{(B)}} \cap k[\bm{y}_{\bm\infty}] = \{0\},\vspace{-0.08in}
    \]
    where \vspace{-0.03in}
    \begin{itemize}[leftmargin=0.5cm,itemsep=0.02cm]
    \item 
    $B = \sum\limits_{0 \leqslant i \leqslant j \leqslant m} D_j^{2(2^{i} - 1)},$
    \item 
$\bm{x} := (x_1, \ldots, x_{s})$, $\bm{y} :=(y_1, \ldots, y_{t})$,
    \item
    $D_j$ is the degree of the \hyperlink{def_ideals}{equidimensional component} of $\langle F\rangle$ of dimension $j$ in $k(\bm{y}_{\bm\infty})[\bm{x}_{\bm{\alpha}}]$, 
    \item we use the convention $0^0 = 0$.
    \end{itemize}
\end{theorem}

For example, if $m=0$, then $B=D_0^0$. If $m=1$, then $B=D_1^2+1+D_0^0$.

\begin{theorem}[Bound for full elimination] \label{th:FullElimination}
For all integers $s, t \geqslant 0$,  
    tuples $\bm{\alpha} = (\alpha_1, \ldots, \alpha_{s})\in \mathbb{Z}_{\geqslant 0}^{s}$, $\bm{\beta} = (\beta_1, \ldots, \beta_{t})\in \mathbb{Z}_{\geqslant 0}^{s}$,  and  $F\subset k[\hyperlink{diff_poly}{\bm{x}_{\bm{\alpha}}}, \hyperlink{diff_poly}{\bm{y}_{\bm{\beta}}}]
    $,  
 \[\sqrt{\langle F \rangle^{\hyperlink{def_ideals}{(\infty)}}} \cap k[\bm{y}_{\bm{\beta}}] =  \sqrt{\langle F\rangle^{ \hyperlink{def_ideals}{(B)}}} \cap k[\bm{y}_{\bm{\beta}}],\vspace{-0.05in}\]
 where
     \begin{itemize} [leftmargin=0.5cm,itemsep=0.02cm] 
     \item $B=\begin{cases}  
	d^{(|\bm{\alpha}| + |\bm{\beta}| - m + 1) 2^{m + 1}} ,& d \geqslant 2,\\
	m+1 , &d=1,
\end{cases}$

\item 
$\bm{x} := (x_1, \ldots, x_{s})$, $\bm{y} :=(y_1, \ldots, y_{t})$,

\item $|\bm{\alpha}|=\alpha_1 + \ldots + \alpha_{s}$, $|\bm{\beta}| = \beta_1 + \ldots + \beta_t$,
     \item $d= \max\limits_{f \in F} \deg f$,
       \item
$m=\dim\langle F\rangle$ in $k[\bm{x}_{\bm{\alpha}}, \bm{y}_{\bm{\beta}}]$.
\end{itemize}    
\end{theorem}

\begin{remark}\label{rem:tight}
	In Theorems~\ref{th:elimination_variables} and~\ref{th:FullElimination}, the expressions for the value of $B$ can be replaced by the tighter ones obtained
    in  inequality~\eqref{inequ:TighterBound}.
\end{remark}

\begin{proposition}[Lower bound for elimination]\label{prop:lower_bound}
  For every positive integer $d$, there exists an irreducible polynomial $P \in \mathbb{Q}[x, y]$ of degree at most $d$ such that \vspace{-0.05in}
  \begin{align*}
    1 &\in \big\langle x' - 1, y' - y, P(x, y) \big\rangle^{\hyperlink{def_ideals}{(\infty)}},\\
    1 &\not\in \big\langle x' - 1, y' - y, P(x, y) \big\rangle^{\hyperlink{def_ideals}{(B - 1)}},
  \end{align*}
  where $B = \binom{d + 2}{2} - 1 = \frac{d (d + 3)}{2}$.
\end{proposition}
\begin{corollary}\label{cor:1}
 The bound in Theorem~\ref{thm:radical_elimination} is asymptotically tight for $m\leqslant 1$.
\end{corollary}

\begin{remark_num}\label{rem:count}
  Consider $F \subset k[\bm{y}_{\bm{\beta}}]$ with $\bm{\beta} \in \mathbb{Z}_{\geqslant 1}^t$.
  One can show that\vspace{-0.05in}
  \begin{equation}\label{eq:finite}
    \langle F \rangle^{(\infty)} \text{ has finitely many solutions} \iff \langle F \rangle^{(\infty)} \cap k[\bm{y}_{\bm{1}}] \text{ has finitely many solutions},
  \end{equation}
  where $\bm{1}=(1, \ldots, 1)$. 
  Furthermore, Theorem~\ref{th:FullElimination} implies that $\langle F \rangle^{(\infty)} \cap k[\bm{y}_{\bm{1}}]$ is finite if and only if $\langle F \rangle^{(B)} \cap k[\bm{y}_{\bm{1}}]$ is finite.
  Thus, using~\eqref{eq:finite} and Theorem~\ref{th:FullElimination}, one can design the following prolongation-relaxation algorithm for counting solutions of a system $F$ of DAEs as follows:\vspace{-0.05in}
  \begin{enumerate}[leftmargin=0.5cm,itemsep=-0.05cm]
    \item Let $B$ be the bound given by Theorem~\ref{th:FullElimination} applied $F \subset k[\bm{y}_{\bm{\beta}}]$.
    \item\label{step2} Successively taking $N$ to be each integer from $1$ to $B$, we check whether $\dim \big(\langle F \rangle^{(N)} \cap k[\bm{y}_{\bm{1}}]\big) \leqslant 0$ and, if it is, stop and go to Step~\ref{step3}.
    \item\label{step3} If, for all $N$ from Step~\ref{step2}, $\dim \big(\langle F \rangle^{(N)} \cap k[\bm{y}_{\bm{1}}]\big) > 0$, return $\infty$.
    Otherwise, we return the number of common zeros  of the polynomials $\langle F \rangle^{(N)} \cap k[\bm{y}_{\bm{\beta}}]$ 
    that are also solutions of $F=0$ as a system of DAEs.
  \end{enumerate}
\end{remark_num}

\section{Examples}\label{sec:examples_numbers}
In this section, we will show
how our bounds can be used for elimination of unknowns in DAEs in practice.  
Our approach is general rather than ad hoc.
Examples~\ref{ex:LV},~\ref{ex:vanderpol}, and~\ref{ex:pendulum} are from modeling, and $m=0,1$ in all of them (cf. Corollary~\ref{cor:1}). We have constructed Example~\ref{ex:control} to show elimination for $m=2$.
All of the computational results below can be reproduced using our {\sc Maple} code at~\url{https://github.com/pogudingleb/DifferentialElimination/tree/master/examples}.
The computation takes less than~$30$ seconds on a laptop. 

\begin{example}[Lotka-Volterra model]\label{ex:LV}
Consider the classical Lotka-Volterra equations (also known as the predator-prey equations),
\begin{equation}\label{eq:LV}
  \begin{cases}
    x' = \alpha x-\beta xy,\\
    y' = \delta xy-\gamma y,
  \end{cases}
  \end{equation}
 in which $x$ and $y$ are the populations of prey and predators, respectively.
  Frequently, one of these quantities, say $y$, cannot be measured in experiments.
  Using our main result, we can determine if there are relations among the parameters $\alpha, \beta, \gamma, \delta$ and the derivatives of $x$ (the population of prey).
 Such relations can be further used to test the model against experimental data~\cite{HHM16}.
 Finding the relations is the problem of eliminating $y$. In this case, we consider~\eqref{eq:LV} in $\mathbb{Q}(\alpha,\beta,\gamma,\delta)(x_{\bm{\infty}})[y,y']$, which defines an affine variety of dimension zero ($m=0$). Therefore, the bound provided by Theorem~\ref{thm:radical_elimination} is $B=1$. The desired relation is
  \[
xx''-x'^2+x(\alpha x-x')(\delta x-\gamma)=0.
  \]
  Note that $\beta$ does not appear in this relation as it is independent of $\alpha,\gamma,\delta,x,x',\ldots$ (cf. \cite[Example~2.13]{HOPY}).
\end{example}

\begin{example}[Van der Pol oscillator]\label{ex:vanderpol}
  The system\vspace{-0.1in}
  \begin{equation}\label{eq:vanderpol}
  \begin{cases}
    y' = z,\\
    (1 - y^2) z - y = 0,
  \end{cases}
  \end{equation}
  is a limiting case of the Van der Pol oscillator~\cite[Example~1.7]{DAEKM}.
  Consider the problem of eliminating $y$.
  System~\eqref{eq:vanderpol} is a system of a linear equation in $y'$ with coefficients in $\mathbb{C}(z)$ and a quadratic equation in $y$ with nonzero discriminant and with coefficients in $\mathbb{C}(z)$.
  Thus, \eqref{eq:vanderpol} defines a zero-dimensional radical ideal in $\mathbb{C}(z)[y, y']$, and so $m = 0$.
  Hence, Theorem~\ref{thm:radical_elimination} implies that, if the elimination is possible, it is possible after one prolongation.
  After this one prolongation, one can now find the following consequence of~\eqref{eq:vanderpol} not involving $y$ using only polynomial elimination:\vspace{-0.07in}
  \[
    z'^2 - z' z - 4 z' z^3 + z^4 + 4 z^6 = 0.\vspace{-0.05in}
  \]
\end{example}

%%%%%

\begin{remark_num}\label{rem:prolong}
  In Examples~\ref{ex:control} and~\ref{ex:pendulum}, we use the following observation.
  If $F \subset k(\bm{y}_{\bm{\infty}})[\bm{x}_{\bm{\alpha}}]$ for some $\bm{\alpha} \in \mathbb{Z}_{\geqslant 0}^{s}$ and, for some $f \in F$, we have $f' \in k(\bm{y}_{\bm{\infty}})[\bm{x}_{\bm{\alpha}}]$, then one can consider the equivalent system $\{F, f'\}$ instead.
  For this new system, the bound given by Theorem~\ref{thm:radical_elimination} will be smaller or the same.
  This is not an ad hoc trick and can be 
used in an algorithm.
\end{remark_num}

%%%%%

\begin{example}[Pendulum]\label{ex:pendulum}
  In this example, we will show how our bounds can be used to  show the {\em impossibility} of elimination.
  Consider the following system from~\cite[p.~725]{GP84}:
  \begin{equation}
  \begin{cases}\label{eq:pendulum}
    x'' = Tx + F_1,\\
    y'' = -T y + 1 + F_2,\\
    x^2 + y^2 = 1
  \end{cases}
  \end{equation}
  with added external force $\mathbf{F} = (F_1, F_2)$.
  System~\eqref{eq:pendulum} describes a pendulum with unit mass, length, and gravity.
  The  unknown functions $x$ and $y$ stand for the coordinates and $T$ denotes the string tension.
  Since differentiating the equation $x^2 + y^2 = 1$ twice does not introduce 
  derivatives that do not appear in the system already, following Remark~\ref{rem:prolong}, we extend system~\eqref{eq:pendulum} by
  \begin{equation}\label{eq:pendulum_extra}
  \begin{cases}
    x x' + y y' = 0,\\
    x'^2 + x x'' + y'^2 + y y'' = 0.
  \end{cases}
  \end{equation}
  
  We will consider problems of deriving  differential equations  in  subsets of variables.
  The results are summarized in Table~\ref{table:pendulum} (in all cases, $m=0,1$). The impossibility of elimination  was established using the approach developed in Section~\ref{sec:probability} with probability at least $99\%$.
  \vspace{-0.1in}
  \begin{table}[h]
  \caption{Example~\ref{ex:pendulum}}
  \label{table:pendulum}
  \centering
  \begin{tabular}{|c|c|c|c|c|}
  \hline
    equation in & $(D_0, D_1)$ & bound from Theorem~\ref{thm:radical_elimination} & elimination possible? \\
    \hline
    $x$ & $(0, 2)$ & $5$ & No \\
    \hline
    $y$ & $(0, 2)$ & $5$ & No \\
    \hline
    $x, F_1$ & $(2, 0)$ & 1 & No \\
    \hline
    $y, F_2$ & $(2, 0)$ & $1$ & No \\
    \hline
  \end{tabular}
  \end{table}
  
\end{example}

\begin{example}\label{ex:control} We will now present an example that illustrates that the main result can also be used in $m = 2$ in practice whether or not elimination is possible.
Consider\vspace{-0.1in}
\begin{equation}\label{eq:control}
  \begin{cases}
    x_1' = x_2 + u_1,\\
    x_2' = x_1 + u_2,\\
    x_3' = x_3(u_1 + u_2' - x_3).
  \end{cases}
\end{equation}
One can think of $u_1$ and $u_2$ as control variables whose values can be prescribed in order to achieve a certain behavior for $x_1, x_2, x_3$.
If there is a consequence of~\eqref{eq:control} involving only two of $x_1, x_2, x_3$, say $x_1$ and $x_2$, this would be a natural restriction on the trajectories on the $(x_1, x_2)$-plane that can be achieved.

We consider all three possible pairs of variables to keep $(x_1, x_2)$, $(x_2, x_3)$, and $(x_1, x_3)$.
For the case $(x_1, x_2)$, we additionally observe that one can add the derivative of the second equation from~\eqref{eq:control}\vspace{-0.05in}
\[
  x_2'' = x_1' + u_2'\vspace{-0.05in}
\]
without changing $\bm{\alpha}$ in the application of Theorem~\ref{thm:radical_elimination} (see Remark~\ref{rem:prolong}).

The results are summarized in Table~\ref{table:control} (in all cases, $m=1,2$).
An equation only in $x_2$ and $x_3$ can be found using polynomial elimination  after one prolongation.
The impossibility of elimination in the cases $(x_1, x_2)$ and $(x_1, x_3)$ was established using the approach developed in Section~\ref{sec:probability} with probability at least $99\%$.

\smallskip

\begin{table}[h]
\centering
\caption{Example~\ref{ex:control}}\label{table:control}
\begin{tabular}{|c|c|c|c|}
  \hline
  equation in & $(D_0, D_1, D_2)$ & bound from Theorem~\ref{thm:radical_elimination} & elimination possible? \\ 
  \hline
  $x_1, x_2$ & $(0, 2, 0)$ & $5$ & No \\
  \hline
  $x_1, x_3$ & $(0, 0, 1)$ & $3$ & No \\
  \hline
  $x_2, x_3$ & $(0, 0, 1)$ & $3$ & Yes\\
  \hline
\end{tabular}
\end{table}

\end{example}

\vspace{-0.25in}
%%%%%%%%%%%%%%%%%%%%%%%%%%%%%%%%%%%%%%%%%%%%%%%%%%%%%%%%%%%%%%%%%%%%%%%%%%%%%%%%%%%%%%
\section{Proofs}
 The proofs are structured as follows. We first show, in Section~\ref{sec:triangular_differentiate}, a new method that allows to build a dimension reduction procedure in such a way that the degree of the newly added equation is bounded by the degree of the ideal. In Section~\ref{sec:multdiff}, we establish a relation between differentiation and intersection of ideals, as well as gather results on the Noether exponent we will use later. Using these methods and results, the proof of the bound is finished in Section~\ref{sec:proofs} along the following lines:\vspace{-0.05in}
\begin{itemize}[leftmargin=0.5cm,itemsep=-0.05cm]
\item we obtain a bound for the radical differential ideal membership problem for prime, radical equidimensional radical, and arbitrary polynomial ideals of the equations of the system to prove Proposition~\ref{prop:any};

  \item From Proposition~\ref{prop:any}, we deduce a bound for the elimination problem given in Theorem~\ref{thm:radical_elimination}.
  By estimating the geometric data in terms of the combinatorial data, we deduce bounds for the elimination problem given in Theorems~\ref{th:elimination_variables} and~\ref{th:FullElimination} from Proposition~\ref{prop:any}. 
  
\end{itemize}\vspace{-0.05in}
Our proof of a new lower bound is given in Section~\ref{sec:lowerbound}. 

\hypertarget{variety}{}
For a field $k$, let $\overline{k}$ denote the algebraic closure of $k$.
For $S\subset k[\bm{x}]$, the set of $\overline{k}$-points of the affine variety of $S$ is denoted by $V(S)$.

\vspace{-0.05in}
\subsection{Dimension reduction}\label{sec:triangular_differentiate}
In this section, we will show that, if the intersection with a polynomial subring of $k[\bm{x}_{\infty}]$ of the form $k[\bm{x}_{\bm{\alpha}}]$ and differentiation do not preserve a prime polynomial  ideal, then this is witnessed by a polynomial of degree at most the degree of the ideal (see Lemma~\ref{lem:one_polynomial}). This will be one of the keys in our inductive argument to prove the main result.

\vspace{-0.08in}
\subsubsection{General dimension reduction}
Let $\bm{x} = (x_1, \ldots, x_n)$ and $\bm{1} = (1, \ldots, 1) \in \mathbb{Z}^n$. We will use the following result, which is similar to \cite[Lemma~3.1]{Gustavson}:

\begin{lemma}\label{lem:preprolongation}
   For every $\alpha \geqslant 1$ and  prime ideal $I \subset k[\hyperlink{diff_poly}{\bm{x}_{\alpha \cdot \bm{1}}}]$, \vspace{-0.05in}
    \begin{equation}\label{eq:diff_condition}
    \langle I \cap k[\bm{x}_{(\alpha - 1) \bm{1}}]\rangle^{ \hyperlink{def_ideals}{(1)}} \subset I \implies I = \sqrt{  I^{\hyperlink{def_ideals}{(\infty)}} } \cap k[\bm{x}_{\alpha \cdot \bm{1}}].\vspace{-0.05in}
    \end{equation}
\end{lemma}

\begin{proof}
    %We denote the images of $x_i^{(j)}$ in 
    Let $\pi$ be the canonical homomorphism $\pi\colon k[\bm{x}_{\alpha \cdot \bm{1}}] \to B$, where
    \vspace{-0.05in}\[ 
    B = k[\bm{x}_{\alpha \cdot \bm{1}}] / I \supset A = k[\bm{x}_{(\alpha - 1) \bm{1}}] \big/ (I \cap k[\bm{x}_{(\alpha - 1) \bm{1}}]).\vspace{-0.05in}
    \] 
    We claim that the field of fractions $Q(B)$ of $B$ satisfies the differential condition (see~\cite[p. 1146]{Gustavson}).
    It is sufficient to show that, for every $f \in k[\bm{x}_{(\alpha - 1) \bm{1}}]$ such that 
    $\pi(f) = 0$,
    for the polynomial $g = f^{\prime} \in k[\bm{x}_{\alpha \cdot \bm{1}}]$, the equality
   $\pi(g) = 0$
    holds.
     $\pi(f) = 0$
    implies that $f \in I \cap k[\bm{x}_{(\alpha - 1) \bm{1}}]$, so\vspace{-0.05in} 
    \[
    g \in (I \cap k[\bm{x}_{(\alpha - 1) \bm{1}}])^{(1)} \subset I.\vspace{-0.05in}
    \]
    Hence, $\pi(g) = 0$.
   Thus, by \cite[Theorem 4.10]{Pierce}, there exists an extension $K \supset Q(B)$, where $K$ is a differential field, and the differential structure on $K$ is compatible with that of $Q(A)\subset Q(B)$.
    Consider the differential homomorphism $\varphi\colon k[\bm{x}_{\bm{\infty}}] \to K$ defined by  
    $\varphi(x_i) = \pi(x_i)$,
    $1 \leqslant i \leqslant n$.
    Then, $\Ker\varphi \cap k[\bm{x}_{\alpha \cdot \bm{1}}] = I$, so \vspace{-0.05in}
    \[
    \sqrt{ I^{(\infty)} } \cap k[\bm{x}_{\alpha \cdot \bm{1}}] \subset \Ker\varphi \cap k[\bm{x}_{\alpha \cdot \bm{1}}] = I.\vspace{-0.07in}
    \]
    The inverse inclusion is immediate.
\end{proof}

%%%%%%%%%%%%%%%%%%%%%%

\begin{lemma}\label{lem:prolongation} For every  tuple $\bm{\alpha} \in \mathbb{Z}_{\geqslant 1}^n$ and prime ideal $I \subset k[\hyperlink{diff_poly}{\bm{x}_{\bm{\alpha}}}]$, \vspace{-0.05in}
\[
\langle I \cap k[\bm{x}_{\bm{\alpha} - \bm{1}}]\rangle^{\hyperlink{def_ideals}{(1)}} \subset I \implies I = \sqrt{  I^{ \hyperlink{def_ideals}{(\infty)} } } \cap k[\bm{x}_{\bm{\alpha}}].\vspace{-0.05in}
\]
\end{lemma}

\begin{proof}
	Let $\bm{\alpha} = (\alpha_1, \ldots, \alpha_n)$ and $\alpha = \max(\alpha_1, \ldots, \alpha_n)$, and, for every $i$, set $\delta_i = \alpha - \alpha_i$.
    We introduce new variables $\bm{y} = (y_1, \ldots, y_n)$.
    Let $\varphi\colon k[\bm{x}_{\bm{\infty}}] \to k[\bm{y}_{\bm{\infty}}]$ be the differential homomorphism  defined by $\varphi(x_i) = y_i^{(\delta_i)}$ for all $i$.
    Then $J = k[\bm{y}_{\alpha \cdot \bm{1}}]\cdot\varphi(I)$ is a prime ideal in $k[\bm{y}_{\alpha \cdot \bm{1}}]$.
    Since 
    \[
    k[\bm{y}_{\alpha \cdot \bm{1}}]\cdot\varphi\big( \langle I \cap k[\bm{x}_{\bm{\alpha} - \bm{1}}]\rangle^{(1)} \big) = k[\bm{y}_{\alpha \cdot \bm{1}}] \cdot \langle J \cap k[\bm{y}_{(\alpha - 1) \bm{1}}]\rangle^{(1)},
    \]
    we obtain that $(J \cap k[\bm{y}_{(\alpha - 1) \bm{1}}])^{(1)} \subset J$.	
    Lemma~\ref{lem:preprolongation} implies that $J = \sqrt{  J^{(\infty)} } \cap k[\bm{y}_{\alpha \cdot \bm{1}}]$.
    Then \vspace{-0.05in}\[k[\bm{y}_{\bm{\infty}}] \cdot\varphi\big( \sqrt{  I^{(\infty)} } \big) = \sqrt{  J^{(\infty)} }\implies I = \sqrt{ I^{(\infty)} } \cap k[\bm{x}_{\bm{\alpha}}].\qedhere\]
\end{proof}

%%%%%%%%%%%%%%%%%%%%%%%%%%%

\subsubsection{Finding an equation of degree at most the degree of the ideal to lower the dimension}

For a non-negative integer $D$ and an ideal $J \subset k[z_1, \ldots, z_N]$, let 
$J_D = \langle f\in J\:|\: \deg f \leqslant D\rangle$.

\begin{lemma}\label{lem:nonsingular}
For every non-negative integer $D$ and prime ideal $J \subset k[z_1, \ldots, z_N]$  of degree $D$,
  there is a nonempty open subset $U \subset \hyperlink{variety}{V}(J)$ such that, for every $p \in U$,\vspace{-0.05in}
    \[
    J_{\mathfrak{m}} = {(J_D)}_{\mathfrak{m}}, \text{ where } \mathfrak{m} = I(p).\vspace{-0.05in}
    \]
\end{lemma}

\begin{proof}
  Without loss of generality, we can assume that $z_1, \ldots, z_d$ form a transcendence basis of $k[z_1, \ldots, z_N]$ modulo $J$.
  For every $i$, $d + 1 \leqslant i \leqslant N$, we consider $P_i(z_1, \ldots, z_d, z_i)$, a non-zero algebraic relation among $z_1, \ldots, z_d, z_i$ modulo $J$ of the smallest degree. 
  Since $P_i$ is a defining equation of the Zariski closure of the projection of $V(J)$ to the $(z_1, \ldots, z_d, z_i)$-coordinates, for every $i$, $d + 1 \leqslant i \leqslant N$, $\deg P_i\leqslant D$.
  Let \vspace{-0.05in}
  \[
  P := \frac{\partial P_{d + 1}}{\partial z_{d + 1}} \cdot \ldots \cdot \frac{\partial P_{N}}{\partial z_{N}}.\vspace{-0.05in}
  \]
  Let $U := V(J) \setminus V(P)$.
  Since $P_{d + 1}, \ldots, P_N$ are squarefree, $P$ does not vanish everywhere on $Z(J)$, so $U\ne\varnothing$.
  
  Let $p \in U$ and $\mathfrak{m} := I(p)$.
  The inclusion $J_D \subset J$ implies ${(J_D)}_{\mathfrak{m}} \subset J_{\mathfrak{m}}$.
  On the other hand, since $P$ is the determinant of the Jacobian of $P_{d + 1}, \ldots, P_N$ with respect to $z_{d + 1}, \ldots, z_N$ and $P(p) \neq 0$, the polynomials $P_{d + 1}, \ldots, P_N$ form a system of local parameters of $V(J)$ at $p$.
  Then $P_{d + 1}, \ldots, P_N$ generate $J_{\mathfrak{m}}$ by~\cite[Theorem~2.5, p.~99]{Shafarevich}.
\end{proof}

%%%%%%%%%%%%%%%%%%%%%%%%%%%%%%

\begin{lemma} \label{lem:one_polynomial}
  For every tuple $\bm{\alpha} \in \mathbb{Z}^n_{\geqslant 1}$, if $I \subset k[\hyperlink{diff_poly}{\bm{x}_{\bm{\alpha}}}]$ is a prime ideal such that
 $\langle I \cap k[\bm{x}_{\bm{\alpha} - \bm{1}}]\rangle^{\hyperlink{def_ideals}{(1)}} \not\subset I$, 
  then there exists $g \in \langle I \cap k[\bm{x}_{\bm{\alpha}- \bm{1}}]\rangle^{(1)}$ such that \vspace{-0.08in}
  \[\dim \langle I,g\rangle < \dim I\quad
  \text{and}\quad \deg g \leqslant \deg I.\vspace{-0.05in}\]
\end{lemma}

\begin{proof}
  Let $D := \deg I$.
  The inclusion $k[\bm{x}_{\bm{\alpha} - \bm{1}}] \subset k[\bm{x}_{\bm{\alpha}}]$ corresponds to a projection $\pi$.
  Let $X := V(I)$ and $X_0 := \overline{\pi(X)}$.
 \cite[Lemma~2]{Heintz} implies that $\deg X_0 \leqslant \deg X$.
  Consider any $f \in I \cap k[\bm{x}_{\bm{\alpha} - \bm{1}}]$ such that $f' \notin I$.
  Then $X \setminus V(f')$ is a nonempty open subset of $X$.
  Applying Lemma~\ref{lem:nonsingular} to the prime ideal $I \cap k[\bm{x}_{\bm{\alpha} - \bm{1}}]$, we obtain a nonempty subset $U \subset X_0$.
  Let \[p\in\big(\pi^{-1}(U) \cap V(I)\big) \cap \big(V(I) \setminus V(f')\big).\]
  Lemma~\ref{lem:nonsingular} implies that there are polynomials $g_1, \ldots, g_M \in I \cap k[\bm{x}_{\bm{\alpha} - \bm{1}}]$ of degree at most $D$ and $a_1, \ldots, a_M, b_1, \ldots, b_M \in k[\bm{x}_{\bm{\alpha} - \bm{1}}]$ such that\vspace{-0.08in} 
  \[
  f = \frac{a_1}{b_1}g_1 + \ldots + \frac{a_M}{b_M} g_M,\vspace{-0.06in}
  \]
   and, for all $i$, $1 \leqslant i \leqslant M$, $b_i (\pi(p)) \neq 0$.
  We clear the denominators and obtain
  \[
  b_1\cdot \ldots\cdot b_M\cdot f = c_1\cdot g_1 + \ldots + c_M\cdot g_M
  \]
  for suitable $c_1, \ldots, c_M \in k[\bm{x}_{\bm{\alpha} - \bm{1}}]$.
  We differentiate this equality and obtain
  \[
  (b_1\cdot\ldots\cdot b_M)'\cdot f + (b_1\cdot\ldots\cdot b_M)\cdot f' = \big(c_1'\cdot g_1 + \ldots + c_M'\cdot g_M\big) + \big(c_1\cdot g_1' + \ldots + c_M\cdot g_M'\big).
  \]
  Since $f, g_1, \ldots, g_M$ vanish at $p$, and $b_1\cdot\ldots\cdot b_M \cdot f'$ does not vanish at $p$, at least one of $g_1', \ldots, g_M'$, say $g_1'$, does not vanish at $p$.
  Thus, we can set $g := g'_1$.
\end{proof}

%%%%%%%%%%%%%%%%%%%%%%%%%%%%%

\subsection{Multiplicity and differentiation}\label{sec:multdiff}

\subsubsection{Noether exponent}

For a field $k$, $\bar{k}$  will denote its algebraic closure.

\begin{definition}\hypertarget{noether_exp}{}
  Let $I$ be an ideal in a commutative ring.
  The smallest positive integer $\mu$ (if it exists) such that ${(\sqrt{I})}^\mu \subset I$ is called \emph{the Noether exponent} of $I$.
 The Noether exponent is well-defined for any ideal in a Noetherian ring.
\end{definition}

\begin{lemma}\label{lem:radical_extension}
  Let $I$ be an ideal in a $k$-algebra $A$.
  Then\vspace{-0.05in}
  \[
  \bar{k} \otimes_k \sqrt{I} = \sqrt{\bar{k} \otimes_k I}.\vspace{-0.05in}
  \]
\end{lemma}

\begin{proof}
  Let $I_{\alg} := \bar{k} \otimes_k I$ and $J := \bar{k} \otimes_k \sqrt{I}$.
  Then $J \subset \sqrt{I_{\alg}}$.
  Since $A_{\alg} / J \cong \bar{k} \otimes_k (A / \sqrt{I})$ and $A / \sqrt{I}$ is separable due to~\cite[Chapter~V, \S 15, p.~A.V.122, Theorem~1]{Bourbaki:Algebra2}, $A_{\alg} / J$ is reduced, so $J$ is a radical ideal.
  Let $a \in \sqrt{I_{\alg}}$, then there exists $N$ such that $a^N \subset I_{\alg} \subset J$.
  Since $J$ is radical, we have  $a \in J$, and so $J = \sqrt{I_{\alg}}$.
\end{proof}

\begin{corollary}\label{cor:noether_extension}
  Let $I$ be an ideal in a $k$-algebra $A$ with  \hyperlink{noether_exp}{Noether exponent} $\mu$ and $I_{\alg} := \bar{k} \otimes_k I$.
  Then the Noether exponent of $I_{\alg}$ is   at most $\mu$.
\end{corollary}

\begin{proof}
  By Lemma~\ref{lem:radical_extension}, $\sqrt{I_{\alg}}$ is generated by any set of generators of $\sqrt{I}$, so ${(\sqrt{I_{\alg}})}^\mu \subset I_{\alg}$.
\end{proof}

\begin{lemma} \label{lem:NoetherExponentC}
   Let $\bm{x} = (x_1, \ldots, x_n)$ and $\bm{\alpha} \in \mathbb{Z}_{\geqslant 0}^n$.
	For every prime ideal $I \subseteq k[\bm{x}_{\bm{\alpha}}]$  of degree $D_0$ and every $g \in k[\bm{x}_{\bm{\alpha}}]$ with $\deg g = D_1$, the \hyperlink{noether_exp}{Noether exponent} of $\langle I, g \rangle$ does not exceed $D_0D_1$.
\end{lemma}

\begin{proof} 
	If the ground field is algebraically closed, 
	the lemma follows from  \cite[Corollary~4.6]{Jelonek}.
    The case of not necessarily algebraically closed $k$ follows from the lemma applied to $\bar{k} \otimes_k I$ and Corollary~\ref{cor:noether_extension}.
\end{proof}

%%%%%%%%%%%%%%%

%%%%%%%%%%%%%%%%%%%%%%%%%%%%%%%%%%%%%%%%%%%%%%%%%%%%%%%%%%%%%%%%%%%%%%%%%%%%%%%%%%%%%%%%%%%%%%%%%%%%%%%%%%%

\subsubsection{Differentiation and intersection of ideals}\label{sec:diff_intersect}

The following lemma follows from~\cite[Corollary~5.2]{PogudinJets} (see also~\cite[Theorem~2.2]{GowardSmith}).

\begin{lemma}\label{lem:product}
   Let $\bm{x} = (x_1, \ldots, x_{n})$.
	For all $q, m_1, \ldots, m_q \in \mathbb{N}$ and for all 
	ideals (not necessarily differential) $I_1,\ldots,I_q \subset k[\hyperlink{diff_poly}{\bm{x}_{\bm{\infty}}}]$, 
 	\[
    I_1^{\hyperlink{def_ideals}{(m_1)}} \cdot \ldots \cdot I_q^{\hyperlink{def_ideals}{(m_q)}} \subset \sqrt{ ( I_1 \cdot \ldots \cdot I_q )^{(m_1 + \ldots + m_q)}}.
    \]
\end{lemma}
    
\begin{lemma}  \label{lem:intersection}
    Let $\bm{x} = (x_1, \ldots, x_{n})$.
	For all $q, m_1, \ldots, m_q \in \mathbb{N}$ and for all 
	ideals (not necessarily differential) $I_1,\ldots,I_q \subset k[\hyperlink{diff_poly}{\bm{x}_{\bm{\infty}}}]$, 
 	$$
    I_1^{\hyperlink{def_ideals}{(m_1)}} \cap \ldots \cap I_q^{\hyperlink{def_ideals}{(m_q)}} \subset \sqrt{ ( I_1 \cap \ldots \cap I_q )^{(m_1 + \ldots + m_q)}}.\vspace{-0.1in}
    $$
\end{lemma}

\begin{proof}
We have
    \[
    \left( I_1^{(m_1)} \cap \ldots \cap I_q^{(m_q)} \right)^{q} \subset  I_1^{(m_1)} \cdot \ldots \cdot I_q^{(m_q)}\implies
    I_1^{(m_1)} \cap \ldots \cap I_q^{(m_q)}\subset \sqrt{I_1^{(m_1)} \cdot \ldots \cdot I_q^{(m_q)}}.\]
    Lemma~\ref{lem:product} implies that the latter radical is contained in $\sqrt{( I_1 \cdot \ldots \cdot I_q )^{(m_1 + \ldots + m_q)}}$.
    Thus, 
    \[
    I_1^{(m_1)} \cap \ldots \cap I_q^{(m_q)} \subset \sqrt{( I_1 \cap \ldots \cap I_q )^{(m_1 + \ldots + m_q)}}.\qedhere
    \]
\end{proof}

\subsection{Proofs of the main results}\label{sec:proofs}

Throughout this section, $k$ denotes a differential field and $\bar{k}$ denotes its algebraic closure.
By~\cite[Lemma~II.1]{Kol}, the derivation on $k$ can be extended uniquely to $\bar{k}$.
We introduce\vspace{-0.03in}
\begin{equation}\label{eq:B}
  B(m, D) := \sum\limits_{i = 0}^m D^{2(2^i - 1)}.\vspace{-0.03in}
\end{equation}
 The arguments in this section are structured as follows. We will start by showing that~\eqref{eq:B} is an upper bound for the number of differentiations in the radical differential ideal membership problem for polynomial prime and equidimensional radical ideals of differential polynomials (see Propositions~\ref{prop:prime} and~\ref{prop:equidimensional}, respectively). This bound is adjusted to include arbitrary polynomial ideals of differential polynomials in Proposition~\ref{prop:any}.
 This results in the bound from Theorem~\ref{thm:radical_elimination}, which we explain in Section~\ref{sec:finishingtheproofs}, in which we also finish proving Theorems~\ref{th:elimination_variables} and~\ref{th:FullElimination} by estimating $B$ in~\eqref{eq:B} in terms of $m$, $d$, and $|\bm{\alpha}|$ or $|\bm{\alpha}|$ and $|\bm{\beta}|$, respectively.
 
 \vspace{-0.05in}
\subsubsection{Prime ideals}\label{secsub:prime}

\begin{proposition}\label{prop:prime}
  For every positive integer $n$, tuple $\bm{\alpha} \in \mathbb{Z}_{\geqslant 0}^n$, prime ideal $I \subset \bar{k}[\hyperlink{diff_poly}{\bm{x}_{\bm{\alpha}}}]$,
    and polynomial $f \in \bar{k}[\bm{x}_{\bm{\alpha}}]$,
  we have
      $$f \in \sqrt{  I^{\hyperlink{def_ideals}{(\infty)}} } \iff f \in \sqrt{ I^{ \hyperlink{def_ideals}{( B(m, D) )}} },$$
  where $m = \dim I$, $D = \deg I$, and $B(m, D)$ is defined in~\eqref{eq:B}.
\end{proposition}

We will use the following 
lemma.

 \begin{lemma}\label{lem:polynomial_convexity}
		For all\vspace{-0.05in}
        \begin{itemize}[leftmargin=0.5cm,itemsep=-0.05cm]
        \item $p(x)\in \mathbb{Z}_{\geqslant 0}[x]$  such that $p(0) = 1$ and $\deg p \geqslant 2$,
        \item 
        $S, n\geqslant 1$ and tuples  $(a_1, \ldots, a_n)$ of positive integers such that $\sum\limits_{i = 1}^n a_i = S$,
        \end{itemize} 
        we have\vspace{-0.15in}
		\[\sum\limits_{i = 1}^n p(a_i) \leqslant p(S).\]\end{lemma}

	\begin{proof}
        It is sufficient to prove that, for all $a, b \geqslant 1$,  $$p(a) + p(b) \leqslant p(a + b).$$
		Let $p(x) = 1 + c_1x + \ldots + c_dx^d$, where $d \geqslant 2$  and $c_d >0$.
        We immediately have
        $$
        c_1(a + b) + c_2(a^2 + b^2) + \ldots + c_{d - 1}(a^{d - 1} + b^{d - 1}) \leqslant c_1(a + b) + c_2(a + b)^2 + \ldots + c_{d - 1}(a + b)^{d - 1}.
        $$
        So, it is sufficient to prove that $2 + c_d(a^{d} + b^d) \leqslant 1 + c_d(a + b)^d$. We have
       \[ 
        1 + c_d(a + b)^d \geqslant 1 + c_da^d + c_d\binom{d}{1}ab^{d - 1} + c_db^d \geqslant 2 + c_d(a^d + b^d).\qedhere
       \]
	\end{proof}

\begin{proof}[Proof of Proposition~\ref{prop:prime}]
  We will prove the proposition by induction on $m$.
  The base cases will be $m = 0, 1$.\vspace{-0.05in}
\begin{itemize}[leftmargin=0.4cm,itemsep=0.01cm]
\item\underline{Case $m = 0$} follows from     Lemma~\ref{lem:prolongation}.
\item\underline{Case $m = 1$.} Then $B(m, D) = D^2 + 1$.
  Consider $f \in \sqrt{ I^{(\infty)} } \cap \bar{k}[\bm{x}_{\bm{\alpha}}]$.
        If $(I \cap \bar{k}[\bm{x}_{\bm{\alpha} - \bm{1}}])^{(1)} \subset I$, then Lemma~\ref{lem:prolongation} implies that $f \in I$.
		Otherwise, by Lemma~\ref{lem:one_polynomial}, there exists $g \in (I \cap \bar{k}[\bm{x}_{\bm{\alpha} - \bm{1}}])^{(1)}$ such that 
        \[
        \dim I > \dim \langle I, g\rangle\quad \text{and}\quad \deg g \leqslant D.
        \]
        Let $J=(I, g)$ and $J = Q_1 \cap \ldots \cap Q_s$ be a primary decomposition of $J$. Then 
        \[
        \sqrt{J} = I_1 \cap \ldots \cap I_s , \text{ where } I_j := \sqrt{Q_j} \text{ for } 1\leqslant j \leqslant s.
        \]
        Since $\dim I_j=0$  for every $j$, $V(I_j) = p_j$ for some point $p_j$.
        Let 
        \[
        m_j=\dim_{\bar{k}} \bar{k}[\bm{x}_{\bm{\alpha}}] / Q_j
        \]
        be the multiplicity of $J$ at the point $p_j$. 
       Then $I_j^{m_j} \subset Q_j$.
       Bezout's theorem  \cite[Theorem~7.7, Chapter 1]{Hartshorne} implies that $$m_1 + \ldots + m_s = \deg I \cdot \deg g \leqslant D^2.$$
  	   The inclusions 
       \[
       f \in \sqrt{ I_1^{(1)} \cap \ldots \cap I_s^{(1)} }\quad\text{and}\quad I_1^{m_1} \cdot \ldots \cdot I_s^{m_s} \subset Q_1\cdot\ldots\cdot Q_s \subset J
       \]
       together with Lemma~\ref{lem:product} imply
        \[
        f \in \sqrt{ \left( I_1^{(1)}\right)^{m_1}\cdot \ldots\cdot \left( I_s^{(1)}\right)^{m_s} } \subset \sqrt{ \big( I_1^{m_1}\cdot \ldots\cdot I_s^{m_s}\big)^{(m_1 + \ldots + m_s)} } \subset \sqrt{J^{(D^2)}} \subset \sqrt{I^{(1 + D^2)}} = \sqrt{I^{(B(1, D))}}.
        \]
        
\item\underline{Inductive step for $m > 1$.}
    Consider $f \in \bar{k}[\bm{x}_{\bm{\alpha}}] \cap \sqrt{ I^{(\infty)} }$.
        If $(I \cap \bar{k}[\bm{x}_{\bm{\alpha} - \bm{1}}])^{(1)} \subset I$, then Lemma~\ref{lem:prolongation} implies that $f \in I$.
		Otherwise, by Lemma~\ref{lem:one_polynomial}, there exists $g \in (I \cap \bar{k}[\bm{x}_{\bm{\alpha} - \bm{1}}])^{(1)}$ such that $$\dim I > \dim \langle I, g\rangle\quad \text{and}\quad \deg g \leqslant D.$$
        Consider the minimal prime decomposition of  $\sqrt{\langle I, g\rangle}$: 
        \[
        \widetilde{I} := \sqrt{\langle I, g\rangle} = I_1 \cap \ldots \cap I_s.
        \]
        Then $\dim I_j = m - 1$ for all $1 \leqslant j \leqslant s$.
        Let $D_j := \deg I_j$ for every $1 \leqslant j \leqslant s$.
        \cite[Theorem~7.7, Chapter 1]{Hartshorne} implies that $\sum\limits_{j = 1}^s D_i \leqslant D^2$.
        Since all $\sqrt{ I_1^{(\infty)} }, \ldots, \sqrt{ I_s^{(\infty)} }$ contain $f$, the inductive hypothesis implies that\vspace{-0.1in}
        \[
        f \in \sqrt{ I_1^{(B(m - 1, D_1))} \cap \ldots \cap I_s^{(B(m - 1, D_s))} }.\vspace{-0.1in}
        \]
        By Lemma~\ref{lem:intersection},\vspace{-0.1in}
        \[
        f \in \sqrt{ \left( I_1 \cap \ldots \cap I_s\right)^{(B)} } = \sqrt{ \widetilde{I}^{(B)} },\quad\text{ where } B := \sum\limits_{i = 1}^s B(m - 1, D_i).\vspace{-0.1in}
        \]
Lemma~\ref{lem:NoetherExponentC} implies that $\widetilde{I}^{D^2} \subset (I, g)$.
        Lemma~\ref{lem:product} implies that
        \begin{equation}\label{eq:fIncl}
        f \in \sqrt{ \left( \widetilde{I}^{(B)} \right)^{D^2} } \subset \sqrt{ \langle I, g\rangle^{(D^2B)} } \subset \sqrt{ I^{(D^2B + 1)} }.
        \end{equation}
        $B(m - 1, t)$ considered as a polynomial in $t$ meets the requirements of Lemma~\ref{lem:polynomial_convexity}.
        Applying Lemma~\ref{lem:polynomial_convexity} and using $\sum\limits_{i = 1}^s D_i \leqslant D^2$, we have\vspace{-0.16in}
        \begin{equation}\label{eq:boundDB}
          D^2 B + 1 = D^2 \sum\limits_{i = 1}^s B(m - 1, D_i) + 1 \leqslant D^2 B(m - 1, D^2) + 1 = B(m, D).
        \end{equation}
        Combining~\eqref{eq:fIncl} and~\eqref{eq:boundDB}, we show that $f \in \sqrt{ I^{(B(m, D))} }$.\qedhere
        \end{itemize}
\end{proof}

%%%%%%%%%%%%%%%%%%%

\subsubsection{Radical equidimensional ideals}\label{secsub:radequidim}

\begin{proposition}\label{prop:equidimensional}
  For every positive integer $n$, tuple $\bm{\alpha} \in \mathbb{Z}_{\geqslant 0}^n$, \hyperlink{def_ideals}{radical equidimensional} ideal $I \subset \bar{k}[\hyperlink{diff_poly}{\bm{x}_{\bm{\alpha}}}]$,
    and polynomial $f \in \bar{k}[\bm{x}_{\bm{\alpha}}]$,
  we have
      $$f \in \sqrt{ I^{\hyperlink{def_ideals}{(\infty)}} } \iff f \in \sqrt{ I^{ \hyperlink{def_ideals}{( B(m, D) )}} },$$
  where $m = \dim I$, $D = \deg I$, and $B(m, D)$ is defined in~\eqref{eq:B}.
\end{proposition}

\begin{proof}
        Let $I = I_1 \cap I_2 \cap \ldots \cap I_s$ be the irreducible prime decomposition of $I$.
      Let $D_j := \deg I_j$ for $1 \leqslant j \leqslant s$. 
      Consider $f \in \sqrt{ I^{(\infty)} } \cap \bar{k}[\bm{x}_{\bm{\alpha}}]$.
\begin{itemize}[leftmargin=0.4cm,itemsep=0.01cm]    
 \item     \underline{Case $m > 0$.}
      Since $f \in \sqrt{ I_j^{(\infty)} }$ for all $1 \leqslant j \leqslant s$, 
      \[
      f \in \sqrt{I_1^{(B(m, D_1))} 
      \cap \ldots \cap I_s^{(B(m, D_s))}}.
      \]
      Lemma~\ref{lem:intersection} implies 
      \[
      f \in \sqrt{(I_1 \cap \ldots \cap I_s)^{(B)}}, \text{ where } B = \sum\limits_{i = 1}^s B(m, D_i).
      \]
      $B(m, t)$ as a polynomial in $t$ meets the requirements of Lemma~\ref{lem:polynomial_convexity}.
      Thus, $B \leqslant B(m, D)$.
      
  \item    \underline{Case $m=0$.}
      Since $B(0, D) = 1$, $f \in \sqrt{I_{j}^{(1)}}$ for all $1 \leqslant j \leqslant s$.
      There exists an integer $M$ such that $f^M \in I_{j}^{(1)}$ for all $1 \leqslant j \leqslant s$.
    Lemma~\ref{lem:intersection} implies that 
	\[
    I_{1} \cap \ldots \cap I_{j - 1} \cap I_{j}^{(1)} \cap I_{j + 1} \cap \ldots \cap I_{s} \subset \sqrt{I^{(1)}} \text{ for every } 1 \leqslant j \leqslant s.\vspace{-0.1in}
    \]
	Hence,\vspace{-0.1in}
	\[
    \sum\limits_{j = 1}^{s} I_{1} \cap \ldots \cap I_{j - 1} \cap I_{j}^{(1)} \cap I_{j + 1} \cap \ldots \cap I_{s} \subset \sqrt{I^{(1)}}.
    \]
	The left-hand side of the above inclusion contains the ideal 
	\begin{equation}\label{eq:sum_ideals}
\sum\limits_{j=1}^{s} I_{1} \cap \ldots \cap I_{j-1} \cap \big\langle f^M\big\rangle \cap I_{j + 1} \cap \ldots \cap I_{s} \supset \big\langle f^M\big\rangle\cdot\sum\limits_{j=1}^{s} I_{1} \cap \ldots \cap I_{j - 1} \cap I_{j + 1} \cap \ldots \cap I_{s}.
	\end{equation}
    Since $\sum\limits_{j=1}^{s} I_{1} \cap \ldots \cap I_{j - 1} \cap I_{j + 1} \cap \ldots\cap I_{s}$ is a sum of zero-dimensional ideals without a common zero, it is equal to  $\bar{k}[\bm{x}_{\bm{\alpha}}]$. 
	Therefore, the right-hand side of~\eqref{eq:sum_ideals} contains $f^M$.
	Thus $f \in \sqrt{I^{(1)}} = \sqrt{I^{(B(0, D))}}$.\qedhere
    \end{itemize}
\end{proof}

%%%%%%%%%%%%%%%%%%%

\subsubsection{Arbitrary ideals}\label{secsub:arbitrary}

\begin{proposition}\label{prop:any}
    For every positive integer $n$, tuple $\bm{\alpha} \in \mathbb{Z}_{\geqslant 0}^n$, ideal $I \subset k[\hyperlink{diff_poly}{\bm{x}_{\bm{\alpha}}}]$,
    and 
    $f \in k[\bm{x}_{\bm{\alpha}}]$,
  we have
      \[
      f \in \sqrt{ I^{\hyperlink{def_ideals}{(\infty)}} } \iff f \in \sqrt{ I^{\hyperlink{def_ideals}{( B )}} },\vspace{-0.15in}
      \]
      where
      \begin{itemize}[leftmargin=0.5cm,itemsep=-0.05cm]
        \item $m = \dim I$,
        \item
$D_i$ is the degree of the \hyperlink{def_ideals}{equidimensional component} of $I$ of dimension  $i$, $0\leqslant i\leqslant m$,
        \item $\mu$ is the \hyperlink{noether_exp}{Noether exponent} of $I$ (which exists because $k[\bm{x}_{\bm{\alpha}}]$ is Noetherian),
        \item $B = \mu\cdot\sum\limits_{i = 0}^m B(i, D_i)$, where $B(m, D)$ is defined in~\eqref{eq:B}. 
      \end{itemize}
  \end{proposition}

\begin{proof}
    We will first prove the proposition for an algebraically closed $k$.
    Consider $f \in \sqrt{ I^{(\infty)} } \cap k[\bm{x}_{\bm{\alpha}}]$.
    For each $0 \leqslant i \leqslant m$, let $I_{i}$ be the radical ideal corresponding to the equidimensional component of dimension $i$ of $I$. %$V(I)$.
Proposition~\ref{prop:equidimensional} implies that $f \in \sqrt{I_{i}^{(B(i,D_i))}}$ for every $0 \leqslant i \leqslant m$.
	Lemma~\ref{lem:intersection} implies that 
	\[
    f \in \sqrt{ \left( I_0 \cap I_1 \cap \ldots \cap I_m \right)^{(S)} } 
	= \sqrt{\left( \sqrt{I} \right)^{(S)}}, \; \text{ where } S = \sum\limits_{i = 0}^{m} B(i,D_i).
    \]
    Since $\left( \sqrt{I} \right)^{\mu} \subset I$, Lemma~\ref{lem:product} implies that $\left(\sqrt{I}\right)^{(S)} \subseteq I^{(\mu \cdot S)}$.
    Hence, $f \in \sqrt{I^{(\mu \cdot S)}} = \sqrt{I^{(B)}}$.
    
    We will finish the proof by considering the case of not necessarily algebraically closed $k$. For an ideal $J \subset k[\bm{x}_{\bm{\alpha}}]$, we denote
   $J_{\alg}=\bar{k}\otimes_k J$.
  Corollary~\ref{cor:noether_extension} implies that the Noether exponent of $I_{\alg}$ is at most $\mu$.
    Then the proposition applied to $I_{\alg} \subset \bar{k}[\bm{x}_{\bm{\alpha}}]$ implies that
    \[
      \sqrt{{(I_{\alg})}^{(\infty)}} \cap \bar{k} [\bm{x}_{\bm{\alpha}}] = \sqrt{{(I_{\alg})}^{(B)}} \cap \bar{k} [\bm{x}_{\bm{\alpha}}].\vspace{-0.1in}
    \]
    Then we have
\begin{align*}
      f \in \sqrt{I^{(\infty)}} \implies f \in \sqrt{{(I_{\alg})}^{(\infty)}} \cap k [\bm{x}_{\bm{\alpha}}] \implies f \in \sqrt{{(I_{\alg})}^{(B)}} \cap k [\bm{x}_{\bm{\alpha}}] &=\sqrt{{\big(I^{(B)}\big)}_{\alg}} \cap k [\bm{x}_{\bm{\alpha}}]\\
      &={\big(\sqrt{I^{(B)}}\big)}_{\alg} \cap k [\bm{x}_{\bm{\alpha}}] \subset \sqrt{I^{(B)}},
 \end{align*}
    where we used Lemma~\ref{lem:radical_extension}.
    Finally, $f \in \sqrt{I^{(B)}} \implies f \in \sqrt{I^{(\infty)}}$ is by definition.
    \end{proof}

%%%%%%%%%%%%%%%%%%

\subsubsection{Bounds for elimination}\label{sec:finishingtheproofs}

\begin{proof}[Proof of Theorems~\ref{th:elimination_variables} and~\ref{thm:radical_elimination}]
Let $J$ be the ideal generated by $F$ in $k(\bm{y}_{\bm{\infty}}) [\bm{x}_{\bm{\alpha}}]$.
\begin{itemize}[leftmargin=0.5cm,itemsep=-0.05cm]
\item
  For each $i$, $0 \leqslant i \leqslant m$, let $D_i$ be the 
 degree of the equidimensional component of dimension $i$  
  of $\sqrt{J}$.
  \item 
	Let $\mu \geqslant 1$ be the Noether exponent of $J$.
    \end{itemize}
     Let 
    \begin{equation}\label{eq:Bth}
      B := \mu\cdot \sum\limits_{i = 0}^m B(i, D_i).
    \end{equation}  
    Then Proposition~\ref{prop:any} implies that
    \[
    1 \in J^{(\infty)} \iff 1 \in J^{(B)}.
    \]
    Thus,
    \begin{equation}\label{eq:elimination}
      \langle F \rangle^{(\infty)} \cap k[\bm{y}_{\bm{\infty}}] \ne\{0\} \iff 1 \in J^{(\infty)} \iff 1 \in J^{(B)} \iff \langle F \rangle^{(B)} \cap k[\bm{y}_{\bm{\infty}}] \ne\{0\} .
    \end{equation}
    
\begin{itemize}[leftmargin=0.4cm,itemsep=0.01cm]    \item\underline{Proof of Theorem~\ref{thm:radical_elimination}.} If $J$ is radical, then $\mu = 1$, so $B = \sum\limits_{i = 0}^m B(i, D_i)$.
    Then the theorem follows from~\eqref{eq:elimination}.
    
\item    \underline{Proof of Theorem~\ref{th:elimination_variables}.}
    To finish the proof, it remains to estimate $B$ in terms of $m$, $d$, and $|\bm{\alpha}|$.
    Let $d_0 := \min\limits_{f \in F} \deg_{\bm{x}} f$ and $r := |F|$.
    Therefore, $d_0 \leqslant d$.
    
        If $d = 1$, then $V(F)$ is an intersection of finitely many hyperplanes.
    Therefore, it is an irreducible variety of dimension $m$ and degree $D_m = 1$.
    Thus, $B = m + 1$.
    
    We will now assume that $d \geqslant 2$.
	By \citep[Corollary~4.6]{Jelonek}, we can bound the Noether exponent by 
	\begin{equation} \label{inequ:BoundNoetherExponent}
	\mu \leqslant d_0 d^{\min \{r, |\bm{\alpha}|\} - 1}.
	\end{equation}
	
For each $i$,    we 
will estimate
$D_i$. 
    By \cite[Lemma~3 and its proof]{JeronimoSabia}, there exist 
    $g_1, \ldots, g_{|\bm{\alpha}| - i} \in k(\bm{y}_{\bm{\infty}}) [\bm{x}_{\bm{\alpha}}]$, where 
    \begin{itemize}[leftmargin=0.5cm,itemsep=-0.05cm]
    \item $g_1$ is the polynomial of minimal degree in $F$, so $\deg g_1 = d_0$, and 
    \item $g_2, \ldots, g_{|\bm{\alpha}| - i}$ are linear combinations of elements of $F$ such that every component of $V(g_1, \ldots, g_{|\bm{\alpha}| - i})$ of dimension greater than $i$ is also a component of $V(F)$.
    \end{itemize}
    Since $V(g_1, \ldots, g_{|\bm{\alpha}|-i}) \supset V(F)$, the above implies that all components of $V(F)$ of dimension $i$ are components of $V(g_1, \ldots, g_{|\bm{\alpha}| - i})$ (but, maybe, there are some superfluous components of dimension $i$ in $V(g_1, \ldots, g_{|\bm{\alpha}| - i})$).
    Since $\deg g_{j} \leqslant d$ for all $j \geqslant 2$, \cite[(8.28) B\'ezout Inequality]{Burgisser} implies that the sum of the degrees of all components of $V(g_1, \ldots, g_{|\bm{\alpha}| - i})$ does not exceed $d_0 d^{|\bm{\alpha}| - i - 1}$.
    Hence, 
	\begin{equation} \label{inequ:BoundTheD_i}
	D_i \leqslant d_0 d^{|\bm{\alpha}| - i - 1}, \, \text{ for every } i=0, \ldots, m.
	\end{equation}		    
    By substituting \eqref{inequ:BoundNoetherExponent} and \eqref{inequ:BoundTheD_i} into \eqref{eq:Bth}, we obtain
    \begin{align} \label{inequ:TighterBound}
		B \leqslant d_0 d^{\min(|\bm{\alpha}|, r) - 1} \cdot \sum\limits_{i = 0}^{m} B\left(i, d_0 d^{|\bm{\alpha}| - i - 1} \right).
    \end{align}
     
 To achieve a simpler formula for the bound, we will replace $d_0$ by $d$.
    In particular, we have
     \begin{align} \label{inequ:N12}
		B \leqslant d^{|\bm{\alpha}|} \cdot \sum\limits_{i = 0}^{m} B\left(i, d^{|\bm{\alpha}| - i} \right) = d^{|\bm{\alpha}|} \cdot \sum\limits_{i = 0}^{m} \sum\limits_{j = 0}^i d^{(|\bm{\alpha}| - i) (2^{j + 1} - 2)}.
    \end{align} 
    Bounding the double sum by a geometric series with common ratio $\frac{1}{d^2}$ twice, we obtain, using $d\geqslant 2$,
    \[
      \sum\limits_{i = 0}^{m} \sum\limits_{j = 0}^i d^{(|\bm{\alpha}| - i) (2^{j + 1} - 2)} \leqslant \frac{d^2}{d^2 - 1} \sum\limits_{i = 0}^{m} d^{(|\bm{\alpha}| - i) (2^{i + 1} - 2)} \leqslant \left( \frac{d^2}{d^2 - 1} \right)^2 d^{(|\bm{\alpha}| - m)(2^{m + 1} - 2)} \leqslant d^{(|\bm{\alpha}| - m)(2^{m + 1} - 2) + 1}
    \]
    Plugging this bound into~\eqref{inequ:N12}, since {$m\leqslant|\bm{\alpha}|-1$}, we obtain
    \begin{equation}\label{eq:bound_final}
      B \leqslant d^{|\bm{\alpha}| + (|\bm{\alpha}| - m)(2^{m + 1} - 2) + 1} \leqslant d^{(|\bm{\alpha}| - m)2^{m + 1} + m} \leqslant d^{(|\bm{\alpha}| - m + 1)2^{m + 1}}.\qedhere
    \end{equation}
    \end{itemize}
\end{proof}

\begin{proof}[Proof of Theorem~\ref{th:FullElimination}]
By applying Proposition~\ref{prop:any} to $I = \langle F \rangle \subset k[\bm{x}_{\bm{\alpha}}, \bm{y}_{\bm{\beta}}]$, we obtain
  \[
  \sqrt{(F)^{(\infty)}} \cap k[\bm{x}_{\bm{\alpha}}, \bm{y}_{\bm{\beta}}] = \sqrt{(F)^{(B)}} \cap k[\bm{x}_{\bm{\alpha}}, \bm{y}_{\bm{\beta}}],  \qquad
    \text{ for } B = \mu  \cdot \sum\limits_{i=0}^{m} { \sum\limits_{j=0}^{i} {D_{i}^{2^{j+1}-2}} },
    \]
    where $\mu$ is the Noether exponent of $I$, and $D_i$ is the degree of the equidimensional component of $I$ of dimension  $i$.
By intersecting both sides with $k[\bm{y}_{\bm{\infty}}]$, we obtain 
\[
  \sqrt{(F)^{(\infty)}} \cap k[\bm{y}_{\bm{\beta}}]= \sqrt{(F)^{(B)}} \cap k[\bm{y}_{\bm{\beta}}].
\]
Estimating $B$ the same way we did in the proof of Theorem~\ref{th:elimination_variables} in~\eqref{inequ:TighterBound},~\eqref{inequ:N12}, and~\eqref{eq:bound_final}, we obtain $B \leqslant d^{(|\bm{\alpha}| + |\bm{\beta}| - m + 1) 2^{m + 1}}$.
\end{proof}

%%%%%%%%%%%%%%%%%%%%%%%%%%%%%%%%%

\subsection{ Asymptotic tightness via a lower bound}\label{sec:lowerbound}

In this section, we prove Proposition~\ref{prop:lower_bound}.
We begin with two auxiliary lemmas.

\begin{lemma}\label{lem:lower_bound}
Let $\bm{x} = (x_1, \ldots, x_n)$.  For all $g_1, \ldots, g_m \in \mathbb{C}[\hyperlink{diff_poly}{\bm{x}_{\bm{\infty}}}]$, positive integers $N$, and formal power series 
  $f_1(t), \ldots, f_n(t) \in \mathbb{C}[\![t]\!]$,
  \[
\big(\forall i\;\; g_i(f_1(t), \ldots, f_n(t)) = \mathrm{O}(t^N),\; t\to 0\big)\implies 1 \notin \bigl\langle g_1, \ldots, g_m \bigr\rangle^{\hyperlink{def_ideals}{(N - 1)}}.
  \]
\end{lemma}

\begin{proof}
  Consider the $\mathbb{C}$-algebra homomorphism $\varphi \colon \mathbb{C}[\bm{x}_{\bm{\infty}}] \to \mathbb{C}$ defined by $\varphi\big(x_i^{(j)}\big) := f_i^{(j)}(0)$.
   Then $1 \notin \operatorname{Ker}\varphi$.
  We will prove the lemma by showing that
  \[
    \bigl\langle g_1, \ldots, g_m \bigr\rangle^{(N - 1)} \subset \operatorname{Ker}\varphi.
  \]
  The chain rule implies that, for every $i$, $1 \leqslant i \leqslant m$, and $j \geqslant 0$,
  \[
    g_i^{(j)}(x_1, \ldots, x_n) |_{x_i = f_i(t)} = \bigl( g_i(f_1(t), \ldots, f_n(t)) \bigr)^{(j)}.
  \]
  Then $\varphi(g_i^{(j)})$ is equal to the value of $\bigl( g_i(f_1(t), \ldots, f_n(t)) \bigr)^{(j)}$ at $t = 0$.
  For every $j < N$ and $i$, since 
  \[
  \bigl( g_i(f_1(t), \ldots, f_n(t)) \bigr)^{(j)} = \bigl( \mathrm{O}(t^N) \bigr)^{(j)} = \mathrm{O}(t^{N - j}),\; t\to 0,
  \]
  the value of  $\bigl( g_i(f_1(t), \ldots, f_n(t)) \bigr)^{(j)}$ at $t = 0$ is zero.
  This proves the lemma.
\end{proof}

\begin{lemma}\label{lem:diff_oper}
  Let $d$ be a positive integer and $P(x, y) \in \mathbb{C}[x, y]$ be a polynomial of degree at most $d$.
  If 
  \[
  P(t, e^t) = \mathrm{O}(t^{B + 1}),\; t \to 0, \text{ where } B = \binom{d + 2}{2} - 1,
  \]
  then $P(t, e^t) = 0$, and so $P$ is the zero polynomial.
\end{lemma}

\begin{proof}
The  function $P(t, e^t)$ is a $\mathbb{C}$-linear combination of $\{t^i e^j\:|\:0 \leqslant i + j \leqslant d\}$.
  All these functions are annihilated by the following differential operator
    \[
  D := \left( \tfrac{\partial}{\partial t} \right)^{d + 1} \left( \tfrac{\partial}{\partial t} - 1 \right)^{d} \left( \tfrac{\partial}{\partial t} - 2 \right)^{d - 1} \cdot\ldots\cdot \left( \tfrac{\partial}{\partial t} - d \right)   
  \]
  of order $B + 1$ with constant coefficients, so $D (P(t, e^t)) = 0$.
  Every solution of $D$ is uniquely determined by its first $B + 1$ Taylor coefficients and $0$ is a solution of $D$, so $P(t, e^t) = 0$.
\end{proof}

\begin{proof}[Proof of Proposition~\ref{prop:lower_bound}]

  We will show that
  \begin{equation}\label{eq:inconsistent}
    1 \in \big\langle x' - 1, y' - y, P(x, y) \big\rangle^{(\infty)}
   \end{equation}
  holds for every $P(x, y) \in \mathbb{C}[x, y]$ such that $P(x, 0) \neq 0$.
  Since  system~\eqref{eq:inconsistent} has constant coefficients, it is consistent if and only if it has a solution in $\mathbb{C}[\![t]\!]$ (follows from~\cite[Proposition~3.2 and Corollary~3.6]{Pogudin:DiffNoether}).
  Every solution of $x' = 1, \; y' = y$ in $\mathbb{C}[\![t]\!]$ is of the form $x(t) = t + a$, $y = be^t$ for some $a, b \in \mathbb{C}$.
  If $b \neq 0$, then $P(x(t), y(t)) \neq 0$ due to the algebraic independence of $t$ and $e^t$ over $\mathbb{C}$.
  If $b = 0$, then $x(t)$ is a root of the nonzero polynomial $P(x, 0)$ with constant coefficients. This is impossible.
  
 For $0 \leqslant i + j \leqslant d$, let $f_{i, j} \in \mathbb{Q}[t]$ be the truncation of the power series $t^i \cdot e^{jt}$ to the degree $B$.
  If the polynomials $\{f_{i, j}\:|\:0 \leqslant i + j \leqslant d\}$ were linearly dependent over $\mathbb{Q}$,
  there would exist $\lambda_{i, j} \in \mathbb{Q}$ not all zeros for $0 \leqslant i + j \leqslant d$ such that
  \[
  f := \sum\limits_{0 \leqslant i + j \leqslant d} \lambda_{i, j} t^i e^{jt} = \mathrm{O}(t^{B + 1}),\ \ t\to 0.
  \] 
  The power series $f$ is nonzero due to the algebraic independence of $t$ and $e^t$.
  On the other hand, Lemma~\ref{lem:diff_oper} implies that $f = 0$.
  The obtained contradiction implies that $\{f_{i, j}\:|\:0 \leqslant i + j \leqslant d\}$ are linearly independent over $\mathbb{Q}$.
  
  Since there are $B + 1$ of the $f_{i, j}$, they form a basis of the $\mathbb{Q}$-vector space of polynomials of degree at most $B$.
  Thus, there exist $\mu_{i, j} \in \mathbb{Q}$ for $0 \leqslant i + j \leqslant d$ such that
  \[
  g(t) := \sum\limits_{0 \leqslant i + j \leqslant d} \mu_{i, j} t^i e^{jt} = t^B + \mathrm{O}(t^{B + 1}),\ \ t\to 0.
  \]
  We define\vspace{-0.05in}
  \[
    P(x, y) := \sum\limits_{0 \leqslant i + j \leqslant d} \mu_{i, j} x^i y^j.
  \]  
  We claim that $P(x, y)$ is irreducible.
  Assume the contrary, so $P(x, y) = P_1(x, y) P_2(x, y)$, where $\deg P_1 = d_1\geqslant 1$ and $\deg P_2 = d_2\geqslant 1$.
  Then there exist integers $B_1$ and $B_2$ such that $B_1 + B_2 = B$ and \vspace{-0.05in}
  \[
  P_i(t, e^t) = t^{B_i} + \mathrm{O}(t^{B_i + 1}), \; t \to 0\ \text{ for } i = 1, 2.\vspace{-0.05in}
  \]
  Since \vspace{-0.05in}\[B = \tfrac{(d + 3)d}{2} > \tfrac{(d_1 + 3)d_1}{2} + \tfrac{(d_2 + 3)d_2}{2},\] we have $B_i > \frac{(d_i + 3)d_i}{2}$ for some $i$, say for $i = 1$.
  Then Lemma~\ref{lem:diff_oper} applied to polynomial $P_1$ of degree $d_1$ implies that $P_1(t, e^t) = 0$.
  The obtained contradiction proves that $P(x, y)$ is irreducible.
  
Since $P(x, y)$ is irreducible, $P(x, 0)$ is not zero.
  This implies that $x' - 1 = y' - y = P(x, y) = 0$ is inconsistent.
  Lemma~\ref{lem:lower_bound} applied to\vspace{-0.07in} \[g_1 = x' - 1,\ g_2 = y' - y,\ g_3 = P(x, y),\ f_1(t) = t,\ f_2(t) = e^t,\ \text{and}\ N = B\vspace{-0.1in}\] implies \vspace{-0.05in}
  \[
    1 \not\in \big\langle x' - 1, y' - y, P(x, y) \big\rangle^{(B - 1)}.\qedhere
  \]
\end{proof}

\begin{example}
  Based on the proof of Proposition~\ref{prop:lower_bound}, one can generate polynomial $P$ using only linear algebra.
  For example, for $d = 2$, we obtain\vspace{-0.05in}
  \[
    P(x, y) = -2 x^2 - 8 x y + y^2 - 10 x + 16 y - 17.\vspace{-0.05in}
  \]
\end{example}

%%%%%%%%%

\begin{corollary}
 The bound in Theorem~\ref{thm:radical_elimination} is asymptotically tight for $m\leqslant 1$.
\end{corollary}
\begin{proof}
Let $I = \bigl\langle x' - 1, y' - y, P(x, y) \bigr\rangle$. We have\vspace{-0.05in} \[A := \mathbb{C}[x,y,x',y']/I\cong \mathbb{C}[x,y]/\langle P(x,y)\rangle.\vspace{-0.05in}\] Since $P$ is irreducible, $A$ is an integral domain, and so the ideal  $I$  is prime and, therefore, radical.
Finally, $m = \dim (F) = 1$ and $D_1 = \deg I=\deg P$, and we have
\vspace{-0.05in}
\[
  D_1^2 + 1 \leqslant 2\cdot\tfrac{D_1(D_1+3)}{2} .\vspace{-0.05in}
\]
For $m=0$, one can take the system $x=0$, $x'-1=0$. Then $1 \in \langle x,x'-1\rangle^{(1)}$ and $1 \notin \langle x,x'-1\rangle$.
\end{proof}

%%%%%%%%%%%%%%%%%%%%%%%%%%%%%%%%%%%%%%%%%%%%%%%%%

%%%%%%%%%%%%%%%%
\vspace{-0.2in}
\section{Dominance of projections of affine varieties and elimination in DAEs}
\label{sec:probability}\vspace{-0.08in}
In this section we will address the problem of verifying whether the projection of an affine variety to an affine subspace is Zariski dense  by analyzing the fibers of the projection. We will then connect this with an algorithm that verifies whether it is possible to eliminate 
a set of unknowns  \vspace{-0.05in}
\begin{itemize}[leftmargin=0.5cm,itemsep=-0.05cm]
\item in a system of  polynomial equations (see Section~\ref{subsec:probability_polynomial}) and
\item  as a consequence of our main result, in a system of DAEs (see Section~\ref{subsec:probability_dae}).
\end{itemize}
\vspace{-0.23in}
\subsection{Dominance of projections of affine varieties}
\label{subsec:probability_polynomial}
\vspace{-0.05in}
The possibility of elimination of a subset of unknowns for polynomial systems is equivalent to the dominance of the corresponding projection of affine varieties.
Verifying whether the  projection of an affine variety to an affine subspace is Zariski dense can be done by, for example, calculating Gr\"obner bases with respect to elimination monomial orderings. However, this could be very time-consuming. 
One can try the following naive approach:
\begin{itemize}[leftmargin=0.5cm,itemsep=-0.05cm]
\item Consider the affine variety \[xy-1=0,\]
whose projection to the $y$-line is dominant. 
What if we consider the fiber of the projection over, say, $y = a$? Note that $xa-1 = 0$ defines a non-empty variety 
if and only if $a \ne 0$.
\item Consider the affine variety
\[ x+y=0,\ \ x=0,
\]
whose projection to the $y$-plane is the point $\{0\}$, and so is not dominant. What if we again consider the fiber over $y = a$? In this case, $x+a=0$, $x=0$ defines an empty variety if and only if $a \ne 0$.
\end{itemize}
What we see in each of the above examples that, for all $a\ne 0$,
\begin{equation}\label{eq:equivalencedomcons}\text{the projection to the $y$-line is {\em dominant}} \iff \text{the fiber over $a$ of the projection is {\em nonempty}}.\end{equation}
Hence, for every field $k$ and every finite subset $S \subset k$, \[|\{a\in S\:|\:\eqref{eq:equivalencedomcons}\ \text{holds}\}| \geqslant |S|-1. \]
We will now show how to generalize this idea to arbitrary affine varieties bounding the size of the exceptional set of points $\bm{a}$ in a finite grid in
$\mathbb{A}^r$ for which the dominance of a projection of an affine variety to $\mathbb{A}^r$ is not equivalent to the emptiness of the fiber over $\bm{a}$.

\begin{proposition}\label{prop:probelim}
  For every
  \begin{itemize}[leftmargin=0.5cm,itemsep=-0.05cm]
  \item affine variety $X \subset {\mathbb{A}^q}\times{\mathbb{A}^r}$  and
  \item  finite subset $S \subset k$,
   \end{itemize} 
  the number of points $\bm{a}=(a_1,\ldots,a_r)\in S^r\subset\mathbb{A}^r$ such that
  \[
  \text{the projection of $X$ to $\mathbb{A}^r$ is  dominant} \iff \text{the fiber over $\bm{a}$ of the projection is  nonempty}
  \]
  is at least
  \[N:=\left(1-\frac{\deg X}{|S|}\right)\cdot |S|^r.\]
\end{proposition}

\begin{proof}
  Let $\pi \colon \mathbb{A}^q\times\mathbb{A}^r\to \mathbb{A}^r$ be the projection. 
  Assume that $\overline{\pi(X)} \neq \mathbb{A}^{r}$.
 \cite[Lemma~2]{Heintz} implies that $\deg \overline{\pi(X)} \leqslant \deg X$, so there exists a polynomial $P_1 \in k[y_1, \ldots, y_r]$ of degree at most $\deg X$~\cite[Proposition~3]{Heintz} such that $\overline{\pi(X)} \subset V(P_1)$.
  Thus, \[
  P_1(\bm{a}) \neq 0 \implies \pi^{-1}(\bm{a}) \cap X = \varnothing.
  \]
 Due to the Demillo-Lipton-Schwartz-Zippel lemma (see~\cite[Proposition~98]{Zippel}), $P_1(\bm{a}) \neq 0$ for at least 
  \[
    \left(1 - \frac{\deg P_1}{|S|}\right)\cdot|S|^r \geqslant N
  \]
  many points 
   $\bm{a}\in S^r$.
  Assume that $\overline{\pi(X)} = \mathbb{A}^{r}$.
  Then there exists an irreducible component $Z \subset X$ such that $\overline{\pi(Z)} = \mathbb{A}^r$.
  \cite[Lemma~4.4]{HOPY} implies that there exists a proper subvariety $Y \subset \mathbb{A}^r$ such that
  \begin{itemize}[leftmargin=0.5cm,itemsep=-0.05cm]
    \item $\deg Y \leqslant \deg Z$;
    \item for every $p \in \mathbb{A}^r \setminus Y$, $\pi^{-1}(p) \cap Z \neq \varnothing$.
  \end{itemize}
  Then there exists a polynomial $P_2 \in k[y_1, \ldots, y_r]$ with $\deg P_2 \leqslant \deg Y$~\cite[Proposition~3]{Heintz} such that 
  \[
  P_2(\bm{a}) \neq 0 \implies \bm{a} \notin Y \implies \pi^{-1}(\bm{a}) \cap Z \neq \varnothing \implies \pi^{-1}(\bm{a}) \cap X \neq \varnothing. 
  \]
  Due to~\cite[Proposition~98]{Zippel} again, $P_2(\bm{a}) \neq 0$ for at least
  \[\left(1 - \frac{\deg P_1}{|S|}\right)\cdot|S|^r \geqslant N\vspace{-0.05in}
  \]
  many points  
  $\bm{a}\in S^r$.
\end{proof}

\vspace{-0.12in}
\subsection{Connection to elimination of unknowns in polynomial systems and in DAEs}
\label{subsec:probability_dae}

By the B\'ezout theorem, Proposition~\ref{prop:probelim} can be restated as follows.
\begin{proposition}\label{prop:probelimalg}
  Let\vspace{-0.03in}
  \begin{itemize}[leftmargin=0.5cm,itemsep=-0.05cm]
  \item $f_1, \ldots, f_\ell \subset k[x_1, \ldots, x_q, y_1, \ldots, y_r]$ be polynomials, $\deg f_i\leqslant d$,
  \item
  $0 < p < 1$ be a real number,
  \item  $S \subset k$ with $|S| = \left\lceil\frac{d^{q + r}}{1 - p}\right\rceil
  $,
  \item
   $a_1, \ldots, a_r$ be elements randomly, independently, and uniformly sampled from $S$,
   \item $g_i := f_i|_{y_1 = a_1, \ldots, y_r = a_r}$, $1 \leqslant i \leqslant \ell$.
   \end{itemize}
  Then\vspace{-0.04in}
  \[
  \langle f_1, \ldots, f_\ell \rangle \cap k[y_1, \ldots, y_r] \neq \{0\} \iff 1 \in \langle g_1, \ldots, g_\ell\rangle   
  \vspace{-0.04in}
  \]
  with probability at least $p$.
\end{proposition}
\begin{proof}
  By~\cite[Lemma~3]{JeronimoSabia}, there exist $h_1, \ldots, h_{q + r}$, linear combinations of $f_1, \ldots, f_\ell$, such that $X := V(h_1, \ldots, h_{q + r})$ and $Y := V(f_1, \ldots, f_\ell)$ have the same prime components of dimension~$> 1$.
  Thus, $\deg Y \leqslant \deg X$ and the B\'ezout inequality implies that $\deg X \leqslant d^{q + r}$.
  The proposition follows from Proposition~\ref{prop:probelim} applied to the variety $Y$.
\end{proof}

As a direct consequence, we obtain a Monte Carlo algorithm that verifies if an elimination of unknowns in a system of polynomial equations is possible with probability at least $p$. 
A deterministic algorithm based on similar geometric considerations was designed in~\cite{RSV2018}.
Since degrees of polynomials do not increase under differentiation, using our main result (see Section~\ref{subsec:main_results}), this can be used in a (deterministic or randomized) elimination algorithm for DAEs by\vspace{-0.06in}
\begin{itemize}[leftmargin=0.5cm,itemsep=-0.05cm]
\item calculating the data from the appropriate statements of the main results and then 
\item iterating differentiation and 
(deterministic or randomized) polynomial elimination successively until either an elimination is discovered or the bound from the appropriate main result is reached. 
\end{itemize}\vspace{-0.07in}
Our implementation of a randomized version as well as of a deterministic version is available at~\url{https://github.com/pogudingleb/DifferentialElimination.git}.
%\vspace{-0.05in}
\section*{Acknowledgments}
%\vspace{-0.05in}
The authors are grateful to S.~Gorchinskiy, H.~Hong, R. Hoobler, T. Scanlon,  M.F.~Singer,  and referees for their suggestions.
This work was partially supported by the NSF grants CCF-0952591,  CCF-1563942, DMS-1606334, DMS-1760448, DMS-1853650, DMS-1853482 by the NSA grant \#H98230-15-1-0245, by CUNY CIRG \#2248, by PSC-CUNY grants \#69827-00 47 and 60098-00 48, by the Austrian Science Fund FWF grant Y464-N18, by the strategic program ``Innovatives O\" O 2020'' by the Upper Austrian Government.

\small
\setlength{\bibsep}{1.65pt}
\bibliographystyle{abbrvnat}
\bibliography{bibdata}
\end{document}